\documentclass[11pt,letterpaper]{amsart}
\usepackage{cases}
\usepackage{mathrsfs}
\usepackage{color}
\usepackage{latexsym,amssymb}
\usepackage{a4wide}
\usepackage{amsthm}
\usepackage{amsmath}
\usepackage{graphicx}
\input xy
\xyoption{all}
\usepackage{verbatim}
\usepackage[lite,abbrev,msc-links]{amsrefs}
\usepackage{framed}
\usepackage{bm}

\numberwithin{equation}{section}
\begin{document}
	\theoremstyle{plain}
	\newtheorem{thm}{Theorem}[section]
	\newtheorem{lem}[thm]{Lemma}
	\newtheorem{cor}[thm]{Corollary}
	\newtheorem{cor*}[thm]{Corollary*}
	\newtheorem{prop}[thm]{Proposition}
	\newtheorem{prop*}[thm]{Proposition*}
	\newtheorem{conj}[thm]{Conjecture}
	\theoremstyle{definition}
	\newtheorem{construction}{Construction}
	\newtheorem{notations}[thm]{Notations}
	\newtheorem{question}[thm]{Question}
	\newtheorem{prob}[thm]{Problem}
	\newtheorem{rmk}[thm]{Remark}
	\newtheorem{remarks}[thm]{Remarks}
	\newtheorem{defn}[thm]{Definition}
	\newtheorem{claim}[thm]{Claim}
	\newtheorem{assumption}[thm]{Assumption}
	\newtheorem{assumptions}[thm]{Assumptions}
	\newtheorem{properties}[thm]{Properties}
	\newtheorem{exmp}[thm]{Example}
	\newtheorem{comments}[thm]{Comments}
	\newtheorem{blank}[thm]{}
	\newtheorem{observation}[thm]{Observation}
	\newtheorem{defn-thm}[thm]{Definition-Theorem}
	\newtheorem*{Setting}{Setting}

	\newcommand{\sA}{\mathcal{A}}	\newcommand{\sB}{\mathcal{B}}	\newcommand{\sC}{\mathcal{C}}	\newcommand{\sD}{\mathscr{D}}	\newcommand{\sE}{\mathcal{E}}	\newcommand{\sF}{\mathscr{F}}	\newcommand{\sG}{\mathscr{G}}	\newcommand{\sH}{\mathscr{H}}	\newcommand{\sI}{\mathscr{I}}	\newcommand{\sJ}{\mathscr{J}}	\newcommand{\sK}{\mathscr{K}}	\newcommand{\sL}{\mathscr{L}}	\newcommand{\sM}{\mathcal{M}}	\newcommand{\sN}{\mathscr{N}}	\newcommand{\sO}{\mathcal{O}}	\newcommand{\sP}{\mathscr{P}}
	\newcommand{\sQ}{\mathscr{Q}}	\newcommand{\sR}{\mathscr{R}}	\newcommand{\sS}{\mathcal{S}}	\newcommand{\sT}{\mathscr{T}}	\newcommand{\sU}{\mathscr{U}}	\newcommand{\sV}{\mathscr{V}}	\newcommand{\sW}{\mathscr{W}}	\newcommand{\sX}{\mathcal{X}}	\newcommand{\sY}{\mathcal{Y}}	\newcommand{\sZ}{\mathcal{Z}}	\newcommand{\bZ}{\mathbb{Z}}	\newcommand{\bN}{\mathbb{N}}	\newcommand{\bQ}{\mathbb{Q}}	\newcommand{\bC}{\mathbb{C}}
	\newcommand{\bR}{\mathbb{R}}	\newcommand{\bH}{\mathbb{H}}	\newcommand{\bD}{\mathbb{D}}	\newcommand{\bE}{\mathbb{E}}	\newcommand{\bP}{\mathbb{P}}
	\newcommand{\bV}{\mathbb{V}}	\newcommand{\cV}{\mathcal{V}}	\newcommand{\cF}{\mathcal{F}}	\newcommand{\bfM}{\mathbf{M}}	\newcommand{\bfN}{\mathbf{N}}	\newcommand{\bfX}{\mathbf{X}}	\newcommand{\bfY}{\mathbf{Y}}	\newcommand{\spec}{\textrm{Spec}}	\newcommand{\dbar}{\bar{\partial}}	\newcommand{\ddbar}{\partial\bar{\partial}}	\newcommand{\redref}{{\color{red}ref}}
	
	
	\title[Boundedness of families] {Geometric Shafarevich boundedness conjecture for families of polarized varieties}
	
	\author[Junchao Shentu]{Junchao Shentu}
	\email{stjc@ustc.edu.cn}
	\address{School of Mathematical Sciences,
		University of Science and Technology of China, Hefei, 230026, China}

	\begin{abstract}
		We establish the geometric Shafarevich boundedness conjecture for the moduli stack of stable minimal models, including in particular the moduli stack of KSB pairs.
	\end{abstract}
	
	\maketitle
	\section{Introduction}
	This paper addresses the higher-dimensional generalization of the geometric Shafarevich conjecture over the field of complex numbers, stating that given an integer $g>1$ and a quasi-projective smooth curve $S^o$, the set
	\begin{align}\label{align_Shafa_conj_set}
		\left\{\textrm{family } f:X^o\to S^o\textrm{ of smooth projective curves of genus }g\right\}/\simeq,
	\end{align}
	where $f\simeq g$ when $f,g$ are isomorphic over $S^o$, is a finite set. As a consequence the Hom-stack  
	\begin{align}\label{align_Homstack_Mg}
		\operatorname{Hom}\big((S,D),\,(\overline{\mathcal{M}}_g,\partial\mathcal{M}_g)\big)
	\end{align}  
	consists of finitely many points; here $S$ denotes the smooth compactification of $S^o$, and $D := S \setminus S^o$ is the divisor of boundary points. The conjecture was first established by Par\v{s}in \cite{Parsin1968} in the case $D = \emptyset$, and subsequently generalized by Arakelov \cite{Arakelov1971} to arbitrary $D$. Its arithmetic counterpart-Shafarevich’s conjecture for curves over number fields-was proved by Faltings \cite{Faltings19832}; combined with Par\v{s}in's technique, this result yields a proof of the Mordell conjecture. 
	
	In general, the geometric Shafarevich conjecture fails for families of higher-dimensional varieties due to the existence of non-rigid smooth families-i.e., smooth families that admit nontrivial deformations over the base. Faltings \cite{Faltings1983} constructed explicit examples of non-rigid families of principally polarized abelian varieties;  Liu-Todorov-Yau-Zuo \cite{LTYZ2005} produced non-rigid families of polarized Calabi-Yau manifolds. These counterexamples indicate that, in higher dimensions, the appropriate equivalence relation in (\ref{align_Shafa_conj_set}) is not mere isomorphism over the base, but rather deformation equivalence: smooth families are expected to form \emph{bounded} moduli. Equivalently, the Hom-stack (\ref{align_Homstack_Mg}) is anticipated to be of finite type when $\overline{\mathcal{M}}_g$ is replaced by the compactified moduli stack of higher-dimensional (polarized) varieties. Indeed, boundedness for families of canonically polarized manifolds was established by Bedulev-Viehweg \cite{Viehweg2000} and Kov\'acs-Lieblich \cite{Kovacs2011}; boundedness for polarized Calabi-Yau families  (in a weaker form) was proved by Liu-Todorov-Yau-Zuo \cite{LTYZ2005} and Viehweg-Zuo \cite{VZ2003_2}. Moreover, several sufficient criteria ensuring rigidity of such families have been identified (see \cite{VZ2003_2,Kovacs2005}). For a comprehensive survey of these developments, we refer the reader to \cite{Kovacs2009}.
	
	It is natural to ask whether the geometric Shafarevich boundedness conjecture extends to families of other varieties-a question intrinsically linked to the construction of compactifiable moduli spaces. Prominent examples of compact moduli spaces for algebraic varieties include the modular compactifications of the moduli spaces parametrizing polarized abelian varieties \cite{Alexeev2002}, plane curves \cite{Hacking2004}, canonically polarized manifolds of general type \cite{Kollar1988,Kollar2010,Kollar2023}, polarized Calabi-Yau manifolds \cite{KX2020}, and $K$-polystable Fano manifolds \cite{Xu2020}. Building on the fundamental works of Koll\'ar \cite{Kollar2010,Kollar2023}, Birkar \cite{Birkar2022} introduced the moduli space of stable minimal models to address the fundamental problem of constructing a proper moduli stack parametrizing birational equivalence classes of projective varieties of arbitrary Kodaira dimension-thereby generalizing the Deligne-Mumford's compact moduli of stable curves $\overline{\sM}_{g,n}$ and the moduli stack of KSB pairs \cite{Kollar1988,Kollar2010,Kollar2023}.
	
	Let $\mathcal{M}_{\mathrm{slc}}(d, \Phi_c, \Gamma, \sigma)$ denote the Deligne-Mumford stack parametrizing $(d, \Phi_c, \Gamma, \sigma)$-stable minimal models (c.f. \cite{Birkar2022} or \S \ref{section_moduli}). It is complete and admits a projective coarse moduli space. A central feature of the geometric Shafarevich conjecture for family of curves is the admissibility condition: the family $f \colon X^o \to S^o$ must be everywhere smooth-i.e., it admits no degenerate fibers over $S^o$. Since certain irreducible components of $\mathcal{M}_{\mathrm{slc}}(d, \Phi_c, \Gamma, \sigma)$ may parameterize families consisting exclusively of singular varieties, we introduce the notion of birational admissibility-a condition designed to exclude degenerate fibers in the birational classification framework. A locally closed substack $\mathcal{S} \subset \mathcal{M}_{\mathrm{slc}}(d, \Phi_c, \Gamma, \sigma)$ is \emph{strictly birationally admissible} if the universal family over $\mathcal{S}$ admits a simple normal crossing strict log birational model (see Definition \ref{defn_st_bir_admissible}).
	
	Let $\mathcal{S} \subset \mathcal{M}_{\mathrm{slc}}(d, \Phi_c, \Gamma, \sigma)$ be a strictly birationally admissible locally closed substack. Let $S$ be a reduced algebraic variety. A family of stable minimal models over $S$ is called $\mathcal{S}$-admissible if all of its fibers lie in $\mathcal{S}$. The set of isomorphism classes of $\mathcal{S}$-admissible families over $S$ is in natural bijection with the set $\mathrm{Hom}(S,\mathcal{S})$ of morphisms $S \to \mathcal{S}$. Following \cite{Kovacs2011}, we define an equivalence relation $\simeq_{\mathcal{S}}$ on $\mathrm{Hom}(S,\mathcal{S})$: for $f,g \in \mathrm{Hom}(S,\mathcal{S})$, we write $f \simeq_{\mathcal{S}} g$ if there exists a reduced scheme $T$ of finite type, a morphism $F \colon S \times T \to \mathcal{S}$ (i.e., an $\mathcal{S}$-admissible family over $S \times T$), and two closed points $x_1, x_2 \in T$ such that $f = F|_{S \times \{x_1\}}$ and $g = F|_{S \times \{x_2\}}$. The main results of this paper are the following two theorems.	
	\begin{thm}[Deformation boundedness]\label{thm_main_set}
		Let $S$ be an algebraic variety whose singular locus $S_{\mathrm{sing}}$ is a compact algebraic subset . Let $\mathcal{S} \subset \mathcal{M}_{\mathrm{slc}}(d, \Phi_c, \Gamma, \sigma)$ be a strictly birationally admissible locally closed substack. Then the quotient set
		\[
		\mathrm{Hom}(S,\mathcal{S}) / \simeq_{\mathcal{S}}
		\]
		is finite.
	\end{thm}
	In the language of algebraic stacks, this boundedness statement admits the following geometric refinement:
	\begin{thm}[Boundedness of Hom-stacks]\label{thm_main_Hom}
		Let $S$ be an algebraic variety whose singular locus $S_{\mathrm{sing}}$ is a compact algebraic subset. Let $\overline{S}$ be a compactification of $S$, and let $Z:=\overline{S} \setminus S$. Let $\mathcal{S} \subset \mathcal{M}_{\mathrm{slc}}(d, \Phi_c, \Gamma, \sigma)$ be a strictly birationally admissible locally closed substack, and let $\partial \mathcal{S} := \overline{\mathcal{S}} \setminus \mathcal{S}$ denote its boundary in the closure $\overline{\mathcal{S}} \subset \mathcal{M}_{\mathrm{slc}}(d, \Phi_c, \Gamma, \sigma)$. Then the Hom-stack
		\[
		\mathrm{Hom}\big((\overline{S},Z),\, (\overline{\mathcal{S}}, \partial \mathcal{S})\big)
		\]
		is of finite type.
	\end{thm}
    The admissibility condition is essential for the validity of the boundedness results (see Examples \ref{exmp_1} and \ref{exmp_2}).  
    For the moduli stack $\overline{\mathcal{M}}_{g,n}$ of $n$-pointed stable curves of genus $g$, where $2g - 2 + n > 0$, the boundary divisor $\partial \mathcal{M}_{g,n} := \overline{\mathcal{M}}_{g,n} \setminus \mathcal{M}_{g,n}$ is a simple normal crossing divisor. It induces a canonical stratification of $\overline{\mathcal{M}}_{g,n}$, with the open stratum $\mathcal{M}_{g,n}$ being dense. Each stratum in this stratification parametrizes stable curves of a fixed topological type (i.e., fixed dual graph), and is therefore strictly birationally admissible. When $\mathcal{S} = \mathcal{M}_g$, Theorem \ref{thm_main_set} recovers the classical Shafarevich conjecture for families of smooth projective curves of genus $g$. When $\mathcal{S} = \mathcal{M}_{g,n}$, the theorem yields deformation boundedness for families of smooth $n$-pointed curves of genus $g$. Finally, when $\mathcal{S}$ is a  stratum of the boundary $\partial \mathcal{M}_{g,n}$, Theorem \ref{thm_main_set} implies deformation boundedness for families of $n$-pointed stable curves whose fibers all share the same topological type.
	
	Another illustrative example arises in the KSB compactification of the moduli space of canonically polarized varieties. Let $\mathcal{M}^{\mathrm{sm}}_{d,v}$ denote the moduli stack of smooth $d$-dimensional canonically polarized varieties with canonical volume $v$. This stack is a strictly birationally admissible open substack of its KSB compactification $\mathcal{M}^{\mathrm{KSB}}_{d,v}$. Consequently, Theorem \ref{thm_main_set} implies deformation boundedness for families of smooth canonically polarized $d$-folds of fixed volume $v$-a result originally established by Bedulev-Viehweg \cite{Viehweg2000} and Kov\'acs-Lieblich \cite{Kovacs2011}. More generally, let $\mathcal{M}^{\mathrm{snc}}_{d,v}$ be the moduli stack parametrizing simple normal crossing pairs $(X,D)$ (as in Definition \ref{defn_SNC_family}) satisfying: (i) $K_X + D$ is ample; (ii) $\dim X = d$; and (iii) $\mathrm{vol}(K_X + D) = v$. Then Theorem \ref{thm_main_set} yields deformation boundedness for families of such pairs over any fixed base variety $S$ whose singular locus is compact.
	
	Given a functorial semi-log desingularization procedure (see, e.g., \cite{Bierstone2013} and \cite[\S 10.4]{Kollar2013}), one may construct a stratification of $\mathcal{M}_{\mathrm{slc}}(d,\Phi_c,\Gamma,\sigma)$ whose strata are all strictly birationally admissible (\cite[\S 6]{stjc2025}). Such a stratification is termed a \emph{birationally admissible stratification}. In contrast to the canonical stratification of $\overline{\mathcal{M}}_{g,n}$-which is uniquely determined by topological type-the existence of a birationally admissible stratification is guaranteed, but it is generally non-unique, as it depends on the choice of functorial desingularization. A prototypical example occurs in the KSB compactification of the moduli space of canonically polarized varieties: there, the open substack parameterizing smooth fibers forms a strictly birationally admissible stratum, while the stratification on the boundary arises from the semi-log canonical desingularization of stable pairs.
	
	The proofs of Theorems \ref{thm_main_set} and \ref{thm_main_Hom} rely on an Arakelov-type inequality: a uniform upper bound for the degree of the classifying morphism with respect to a certain ample $\mathbb{Q}$-line bundle on the coarse moduli space $M_{\mathrm{slc}}(d,\Phi_c,\Gamma,\sigma)$.  
	
	Let $f: (X, B), A \to S$ be a family of $(d,\Phi_c,\Gamma,\sigma)$-stable minimal models over a reduced scheme $S$ (c.f. \cite{Birkar2022} or \S \ref{section_moduli}). Since the moduli stack $\mathcal{M}_{\mathrm{slc}}(d,\Phi_c,\Gamma,\sigma)$ is of finite type, for all sufficiently small $a > 0$ and sufficiently large integers $r \gg 1$, the direct image sheaf $f_*\big(r(K_{X/S} + B + aA)\big)$ is locally free and commutes with arbitrary base change. Consequently, the assignment  
	\[
	f \mapsto f_*\big(r(K_{X/S} + B + aA)\big)
	\]  
	defines a locally free coherent sheaf $\Lambda_{a,r}$ on $\mathcal{M}_{\mathrm{slc}}(d,\Phi_c,\Gamma,\sigma)$. Denote its determinant line bundle by $\lambda_{a,r} := \det(\Lambda_{a,r})$. As $\mathcal{M}_{\mathrm{slc}}(d,\Phi_c,\Gamma,\sigma)$ is a Deligne-Mumford stack, some positive integer power $\lambda_{a,r}^{\otimes k}$ descends to a genuine line bundle on the coarse moduli space $M_{\mathrm{slc}}(d,\Phi_c,\Gamma,\sigma)$. We therefore regard $\lambda_{a,r}$ as a $\mathbb{Q}$-line bundle on $M_{\mathrm{slc}}(d,\Phi_c,\Gamma,\sigma)$. Crucially, $\lambda_{a,r}$ is ample for all sufficiently small $a > 0$ and sufficiently large $r$ (\cite{Kovacs2017,Fujino2018}), which ensures the projectivity of $M_{\mathrm{slc}}(d,\Phi_c,\Gamma,\sigma)$.  
	
	The central technical result of this paper is the following uniform Arakelov-type bound on the degree of the classifying morphism with respect to $\lambda_{a,r}$.
	
	\begin{thm}[Arakelov inequality for stable families]\label{thm_main_numbound_moduli}
		Let $f^o \colon (X^o, B^o), A^o \to S^o$ be a birationally admissible family of $(d,\Phi_c,\Gamma,\sigma)$-stable minimal models over a smooth quasi-projective $n$-fold $S^o$, inducing a classifying morphism $\xi^o \colon S^o \to M_{\mathrm{slc}}(d,\Phi_c,\Gamma,\sigma)$. Let $S$ be a smooth projective compactification of $S^o$ such that $D := S \setminus S^o$ is a reduced simple normal crossing divisor, and suppose $\xi^o$ extends to a morphism $\xi \colon S \to M_{\mathrm{slc}}(d,\Phi_c,\Gamma,\sigma)$\footnote{We do not require that $\xi$ factor through $\mathcal{M}_{\mathrm{slc}}(d,\Phi_c,\Gamma,\sigma)$.}. Let $0<a\ll1$ and $r\gg0$, and define  
		\[
		\mathrm{ind}_\xi(\lambda_{a,r}) := \min\big\{k \in \mathbb{Z}_{>0} \mid \xi^*(\lambda_{a,r}^{\otimes k}) \text{ is a line bundle}\big\}.
		\]  
		If $K_S + D$ is pseudo-effective, then for every movable curve class $\alpha \in N_1(S)$ and every $k \in \mathbb{Z}_{>0}$ divisible by $\mathrm{ind}_\xi(\lambda_{a,r})$, we have  
		\begin{align}\label{align_main_Arakelov_inequality}
			c_1\big(\xi^*\lambda_{a,r}\big) \cdot \alpha 
			&\leq \frac{lrd}{2}n^{klrd-1} \, (K_S + D) \cdot \alpha \;+\; \frac{2}{k}\, D \cdot \alpha,
		\end{align}
		where $l = \operatorname{rank}(\Lambda_{a,r})$.  
		If moreover $K_S+D$ is ample, then  
		\begin{align*}
			c_1(\xi^*\lambda_{a,r})\cdot(K_S+D)^{n-1} \leq \frac{lrd}{n}(K_S+D)^{n}.
		\end{align*}		
		In particular, when $\dim S = 1$ and $K_S+D$ is pseudo-effective, we obtain the degree bound  
		\begin{align}\label{align_main_Arakelov_inequlaity_dim1}
		\deg\, \xi^*(\lambda_{a,r}) \leq \frac{lrd}{2} \cdot \deg(K_S + D).
	    \end{align}
	\end{thm}
    The precise quantitative constraints on the parameters $a$ and $r$, which depend only on the numerical invariants $(d,\Phi_c,\Gamma,\sigma)$, are specified in Theorem \ref{thm_Arakelov_family}. Moreover, $\mathrm{ind}_\xi(\lambda_{a,r}) = 1$ whenever the morphism $\xi \colon S \to M_{\mathrm{slc}}(d,\Phi_c,\Gamma,\sigma)$ lifts to the moduli stack $\mathcal{M}_{\mathrm{slc}}(d,\Phi_c,\Gamma,\sigma)$-equivalently, whenever the family $f^o :(X^o,B^o) \to S^o$ extends to a family of stable minimal models over the base $S$.
    Examples \ref{exmp_1} and \ref{exmp_2} demonstrate that the birational admissibility condition is indispensable for the validity of the Arakelov-type inequalities.
    
    The constant appearing in inequality (\ref{align_main_Arakelov_inequality}) is not optimal; a refined version with an improved constant is established in Theorem \ref{thm_Arakelov_family}. In contrast, inequality (\ref{align_main_Arakelov_inequlaity_dim1}) is sharp: equality holds if the classifying morphism of the family factors through a Shimura curve in the moduli space of principally polarized abelian varieties (cf. \cite{VZ2004}).
    
    The proof of Theorem \ref{thm_main_numbound_moduli} relies on the Hodge-theoretic Arakelov inequality established in Theorem \ref{thm_Hodge_Arakelov} and the Viehweg-Zuo type construction of a logarithmic Higgs sheaf associated with an admissible family of stable minimal models, as developed in \cite{stjc2025} (see also Theorem \ref{thm_VZ_construction}). Assume, for simplicity, that $f^o:X^o\to S^o$ is a family of canonically polarized manifolds admitting an extension to a family $f:X\to S$ of KSB-stable varieties. In recent work of Kov\'acs-Taji \cite{Kovacs2024}, the authors construct an embedding  
    \begin{align}\label{align_intro_embed}
    	\xi^\ast\lambda_{0,r}^{\otimes r}\simeq \det f_\ast\big(\mathcal{O}_X(rK_{X/S})\big)^{\otimes r} \hookrightarrow f^{[\mu]}_{\ast}\big(\mathcal{O}_{X^{[\mu]}}(rK_{X^{[\mu]}/S})\big),
    \end{align}  
    valid for sufficiently large $\mu$, where $\mu$ depends only on the discrete invariants $(d,\Phi_c,\Gamma,\sigma)$. Here, $f^{[\mu]}:X^{[\mu]}\to S$ denotes the $\mu$-fold fiber product of $f$ over $S$. The embedding gives rise to an embedding of $\det f_\ast\big(\mathcal{O}_X(rK_{X/S})\big)^{\otimes r}$ into a variation of Hodge structures. One may apply the positivity of the logarithmic cotangent bundle $\Omega_S(\log D)$ (\cite{CP2019}) to derive a higher-dimensional Arakelov-type inequality.
    
    In contrast to \cite{Kovacs2024}, the central technical innovation of this paper consists in replacing the embedding (\ref{align_intro_embed}) with one arising from the alternating sum construction:  
    \[
    \det f_\ast\big(\mathcal{O}_X(rK_{X/S})\big)^{\otimes r} \hookrightarrow \bigotimes^{lr} f_\ast\big(\mathcal{O}_X(rK_{X/S})\big) \simeq f^{[lr]}_{\ast}\big(\mathcal{O}_{X^{[lr]}}(rK_{X/S})\big),
    \]  
    where $l = \operatorname{rank} f_\ast\big(\mathcal{O}_X(rK_{X/S})\big)$. By the geometric characterization of the canonical extension of admissible variations of mixed Hodge structure of geometric origin (\cite{stjc2025}), the embeding give rise to an embedding of $\det f_\ast\big(\mathcal{O}_X(rK_{X/S})\big)^{\otimes r}$ into the canonical extension of a variation of mixed Hodge structure on $S^o$. Since $lr \ll \mu$ in general, this construction leads to a quantitatively sharper Arakelov-type inequality. Furthermore, in addition to the positivity of $\Omega_S(\log D)$, we exploit the parabolic semistability of the Deligne extension of the variation of Hodge structure on $S^o$ (\cite{Simpson1990}) to establish the optimal Arakelov inequality (\ref{align_main_Arakelov_inequlaity_dim1}).
      
	This paper is organized as follows. In Section 2, we explain how an Arakelov-type inequality arises from the existence of a certain variation of mixed Hodge structure (Theorem \ref{thm_Hodge_Arakelov}), building on the generic semipositivity theorem of log cotangent bundle established by Campana-P\v{a}un \cite{CP2019} and Simpson's non-abelian Hodge theory \cite{Simpson1990}. In Section 3, we establish the main Arakelov-type inequality for birationally admissible families of $(d,\Phi_c,\Gamma,\sigma)$-stable minimal models (Theorem \ref{thm_main_numbound_moduli}=Theorem \ref{thm_Arakelov_family}). In Section 4, building on the foundational work of Kov\'acs-Lieblich \cite{Kovacs2011} and M. Olsson \cite{Olsson2007}, we establish the boundedness of birationally admissible families over a fixed base (Theorems \ref{thm_main_set} and \ref{thm_main_Hom}).
	\newline
	\newline
	\textbf{Acknowledgments.} The author is grateful to Professor S\'{a}ndor J. Kov\'{a}cs and Professor Kang Zuo for their interest in this work and for many valuable comments and suggestions.
	\newline
	\newline
	\textbf{Notation and conventions.}
	\begin{itemize}
		\item Let $F$ be a torsion-free coherent sheaf on a smooth algebraic variety $X$. Define its dual by $F^\vee := \mathcal{H}om_{\mathcal{O}_X}(F, \mathcal{O}_X)$; then $F^{\vee\vee}$ is the reflexive hull of $F$. The determinant line bundle $\det(F)$ is defined as the reflexive hull of $\bigwedge^{\operatorname{rank} F} F$.
	\end{itemize}

	
\section{Hodge theoretic Arakelov inequality}
In this section, we establish an Arakelov-type inequality that arises from a certain variation of mixed Hodge structure. The main results of this section are Theorem \ref{thm_Hodge_Arakelov} and Theorem \ref{thm_VZ_Arakelov}.  
\subsection{Variation of mixed Hodge structure}
Let $S$ be a complex manifold. Denote by $\sA^k_S$ the sheaf of $C^\infty$ $k$-forms on $S$, and by $\sA^{p,q}_S$ the sheaf of $C^\infty$ $(p,q)$-forms on $S$.
\begin{defn}  
	A \emph{pre-variation of Hodge structures} $\bV = (\bV_{\bR}, V, \nabla, F^\bullet)$ of weight $m$ on $S$ consists of the following data:  
	\begin{itemize}  
		\item An $\bR$-local system $\bV_{\bR}$ on $S$, together with a holomorphic flat connection $(V, \nabla)$ corresponding to $\bV_{\bC} := \bV_{\bR} \otimes_{\bR} \bC$ via the Riemann-Hilbert correspondence.  
		\item A regular, decreasing filtration $F^\bullet$ of holomorphic subbundles of $V$, such that each graded quotient $F^p / F^{p+1}$ is locally free, and for every point $s \in S$, the fiber $\bV(s) := (\bV_{\bR}(s), F^\bullet(s))$ forms a pure $\bR$-Hodge structure of weight $m$.  
	\end{itemize}  		
	An \emph{$\bR$-polarization} on $\bV$ is a bilinear map between $\bR$-local systems  
	$$ Q : \bV_{\bR} \otimes \bV_{\bR} \to \bR_S(-m), $$  
	which induces an $\bR$-polarization on $\bV(s)$ for every point $s \in S$.  
	
	A \emph{variation of Hodge structures} of weight $m$ is a pre-variation of Hodge structures $\bV = (\bV_{\bR}, V, \nabla, F^\bullet)$ of weight $m$ satisfying the Griffiths transversality condition:  
	$$ \nabla(F^p) \subset F^{p-1} \otimes \Omega_S, \quad \forall p. $$  
\end{defn}
\begin{defn}  
	A \emph{pre-variation of mixed Hodge structures} $\bV = (\bV_{\bR}, V, \nabla, W_\bullet, F^\bullet)$ on $S$ consists of the following data:  
	\begin{itemize}  
		\item An $\bR$-local system $\bV_{\bR}$ on $S$, together with a finite increasing filtration of local subsystems $W_\bullet$ of $\bV_{\bR}$.  
		\item A holomorphic flat connection $(V, \nabla)$ on $S$, along with a finite increasing filtration of sub flat connections $W_{\bullet,\bR}$, such that $(\bV_{\bR}, W_{\bullet,\bR}) \otimes_{\bR} \bC$ corresponds to $(V, \nabla, W_\bullet)$ via the Riemann-Hilbert correspondence.  
		\item A finite decreasing filtration $F^\bullet$ of holomorphic subbundles of $V$,  
	\end{itemize}  
	such that  
	$$
	{\rm Gr}^W_m(\bV) := ({\rm Gr}^W_m(\bV_{\bR}), {\rm Gr}^W_m(V), {\rm Gr}^W_m(\nabla), F^\bullet {\rm Gr}^W_m)
	$$  
	forms a pre-variation of Hodge structures of weight $m$.  
	
	A \emph{variation of mixed Hodge structures} on $S$ is a pre-variation of mixed Hodge structures $\bV = (\bV_{\bR}, V, \nabla, W_\bullet, F^\bullet)$ on $S$ satisfying the Griffiths transversality condition:  
	$$
	\nabla(F^p) \subset F^{p-1} \otimes \Omega_S, \quad \forall p.
	$$  
	
	A \emph{graded $\bR$-polarization} of a variation of mixed Hodge structures $\bV = (\bV_{\bR}, V, \nabla, W_\bullet, F^\bullet)$ consists of an $\bR$-polarization on each ${\rm Gr}^W_m(\bV)$. A variation of mixed Hodge structures is said to be \emph{graded $\bR$-polarizable} if it admits a graded $\bR$-polarization.  
\end{defn}  
\subsection{Admissible variation of mixed Hodge structure}
Let $\overline{S}$ be a complex manifold, and let $E \subset \overline{S}$ be a simple normal crossing divisor. Define $S := \overline{S} \setminus E$. 

Recall that for a flat connection $(V, \nabla)$ on $S$, the lower canonical extension of $(V, \nabla)$ to $\overline{S}$ is a logarithmic connection  
\[
\nabla: \widetilde{V} \to \widetilde{V} \otimes_{\sO_{\overline{S}}} \Omega_{\overline{S}}(\log E),
\]  
which extends $(V, \nabla)$ such that the real parts of the eigenvalues of the residue map along each component of $E$ lie in the interval $[0,1)$. According to \cite[Proposition 5.4]{Deligne1970}, the lower canonical extension always exists and is unique up to isomorphism. Let $\widetilde{W}_m$ denote the lower canonical extension of $W_m$. For the notion of admissibility, we follow Kashiwara \cite{Kashiwara1986}. 
\begin{defn}  
	Let $\bV = (\bV_{\bR}, V, \nabla, W_\bullet, F^\bullet)$ be a graded $\bR$-polarized variation of mixed Hodge structures on the punctured unit disc $\Delta^\ast$. The variation $\bV$ is called \emph{pre-admissible} with respect to $\Delta$ if the following conditions are satisfied:  
	\begin{itemize}  
		\item The monodromy operator of $\bV_{\bR}$ at the origin is quasi-unipotent.  
		\item The logarithm $N$ of the unipotent part of the residue map ${\rm Res}_0(\nabla)$ admits a weight filtration relative to $\widetilde{W}_\bullet(0)$.  
		\item The filtration $F^\bullet$ extends to a filtration $\widetilde{F}^\bullet$ of subbundles of $\widetilde{V}$ such that ${\rm Gr}_{\widetilde{F}}^p {\rm Gr}^{\widetilde{W}}_m \widetilde{V}$ is locally free for all $p$ and $m$.  
	\end{itemize}  
\end{defn}
\begin{defn}  
	Let $\bV = (\bV_{\bR}, V, \nabla, W_\bullet, F^\bullet)$ be a graded $\bR$-polarized variation of mixed Hodge structures on $S$. The variation $\bV$ is called \emph{admissible} with respect to $\overline{S}$ if, for every holomorphic map $f: \Delta \to \overline{S}$ such that $f(\Delta^\ast) \subset S$, the pullback $(f|_{\Delta^\ast})^\ast(\bV)$ is pre-admissible with respect to $\Delta$.  
\end{defn}
\begin{prop} \emph{(Kashiwara \cite[Proposition 1.11.3]{Kashiwara1986})} \label{prop_Kashiwara_extendHodge}  
	Let $\bV = (\bV_{\bR}, V, \nabla, W_\bullet, F^\bullet)$ be a graded $\bR$-polarized variation of mixed Hodge structures on $S$, which is admissible with respect to $\overline{S}$. Then the Hodge filtration $\{F^p\}$ admits an extension $\{\widetilde{F}^p\}$ on $\overline{S}$ such that the following conditions are satisfied:  
	\begin{enumerate}  
		\item $\widetilde{F}^p$ is a subbundle of $\widetilde{V}$ for all $p$.  
		\item ${\rm Gr}_{\widetilde{F}}^p {\rm Gr}_m^{\widetilde{W}}(\widetilde{V})$ is locally free for all $p$ and $m$.  
		\item $\widetilde{F}^\bullet$ satisfies the Griffiths transversality condition:  
		$$
		\nabla(\widetilde{F}^p) \subset \widetilde{F}^{p-1} \otimes \Omega_{\overline{S}}(\log E), \quad \forall p.
		$$  
	\end{enumerate}  
\end{prop}
\subsection{The associated logarithmic Higgs bundle}\label{section_lHB_BM}
Let $\overline{S}$ be a complex manifold and $E \subset \overline{S}$ a reduced simple normal crossing divisor. Set $S := \overline{S} \setminus E$. Let $\mathbb{V} = (\mathbb{V}_\mathbb{R}, V, \nabla, W_\bullet, F^\bullet)$ be an admissible variation of mixed Hodge structure on $S$ with respect to the compactification $(\overline{S}, E)$.  

Denote by $(\widetilde{V}, \nabla)$ (resp. $(\widetilde{W}_\bullet, \nabla)$) the lower canonical extension of $(V, \nabla)$ (resp. $(W_\bullet, \nabla)$) to $\overline{S}$. Let $\widetilde{F}^\bullet$ be the unique extension of $F^\bullet$ to a locally free filtration on $\widetilde{V}$ satisfying Griffiths transversality, as guaranteed by Proposition \ref{prop_Kashiwara_extendHodge}. That is,  
\begin{align}\label{align_Griff_tran_Kashiwara_extension}  
	\nabla(\widetilde{F}^p) \subset \widetilde{F}^{p-1} \otimes_{\mathcal{O}_{\overline{S}}} \Omega_{\overline{S}}(\log E), \quad \forall p.
\end{align}  
For each integer $p$, define the associated graded sheaf  
\[
\widetilde{H}^p := \mathrm{Gr}_{\widetilde{F}}^p(\widetilde{V}) = \widetilde{F}^p / \widetilde{F}^{p+1}.
\]  
Then (\ref{align_Griff_tran_Kashiwara_extension}) implies that $\nabla$ induces a $\mathcal{O}_{\overline{S}}$-linear morphism  
\begin{align*}
	\theta: \widetilde{H}^p \to \widetilde{H}^{p-1} \otimes_{\mathcal{O}_{\overline{S}}} \Omega_{\overline{S}}(\log E), \quad \forall p.
\end{align*}  
Set $\widetilde{H} := \bigoplus_p \widetilde{H}^p$, and let $\theta: \widetilde{H} \to \widetilde{H} \otimes_{\mathcal{O}_{\overline{S}}} \Omega_{\overline{S}}(\log E)$ denote the direct sum of the component maps. The flatness of $\nabla$ (i.e., $\nabla^2 = 0$) implies $\theta^2 = 0$. Hence $(\widetilde{H}, \theta)$ is a logarithmic Higgs bundle on the log pair $(\overline{S}, E)$.  
\begin{defn}\label{defn_LHB_ass_VHSBM}  
	We call $(\widetilde{H}, \theta)$ the \emph{lower canonical logarithmic Higgs bundle} associated with the admissible variation of mixed Hodge structure $\mathbb{V}$.  
\end{defn}  
The logarithmic Higgs bundle $(\widetilde{H}, \theta)$ carries a natural weight filtration by Higgs subbundles, induced from $W_\bullet$. For each integer $m$, define  
\[
W_m(\widetilde{H}^p) := \widetilde{W}_m\big(\mathrm{Gr}_{\widetilde{F}}^p(\widetilde{V})\big),
\]  
and set  
\[
W_m(\widetilde{H}) := \bigoplus_{p \in \mathbb{Z}} W_m(\widetilde{H}^p).
\]  
Then $(W_m(\widetilde{H}), \theta)$ is a logarithmic Higgs subbundle of $(\widetilde{H}, \theta)$, i.e.,  
\[
\theta\big(W_m(\widetilde{H})\big) \subset W_m(\widetilde{H}) \otimes_{\mathcal{O}_{\overline{S}}} \Omega_{\overline{S}}(\log E).
\]  
Passing to graded pieces yields the logarithmic Higgs bundles  
\begin{align}\label{eq_grW_Higgs}  
	\theta: \mathrm{Gr}^W_m(\widetilde{H}) \to \mathrm{Gr}^W_m(\widetilde{H}) \otimes_{\mathcal{O}_{\overline{S}}} \Omega_{\overline{S}}(\log E),
\end{align}  
which coincide with the lower canonical logarithmic Higgs bundles associated with the variations of pure Hodge structure $\mathrm{Gr}^W_m(\mathbb{V})$.
\subsection{non-abelian Hodge theory on algebraic curve}\label{section_nonabelian_Hodge}
In this section, we recall Simpson's non-abelian Hodge theory \cite{Simpson1990}, which plays a foundational role in our proof of the Arakelov inequality.

Let $C$ be a smooth projective curve and $D = \{x_1,\dots,x_d\} \subset C$ a reduced divisor. Let $\mathbb{V} = (\mathcal{V}, \nabla, \mathcal{F}^\bullet, Q)$ be an $\mathbb{R}$-polarized variation of Hodge structure of weight $w$ on $C \setminus D$. Denote its associated graded Higgs bundle by  
\[
(H := \mathrm{Gr}_{\mathcal{F}^\bullet}\mathcal{V},\ \theta := \mathrm{Gr}_{\mathcal{F}^\bullet}\nabla).
\]  
Under Simpson’s correspondence \cite{Simpson1988}, $(H,\theta)$ is the Higgs bundle associated with $(\mathcal{V},\nabla)$. The Hodge metric $h_Q$ induced by the polarization $Q$ is a harmonic metric on $(H,\theta)$, and the triple $(H,\theta,h_Q)$ forms a tame harmonic bundle in the sense of Simpson \cite{Simpson1990}. Moreover, $(H,\theta)$ is a system of Hodge bundles \cite[\S 4]{Simpson1992}: it admits a decomposition  
\begin{align*}
	H = \bigoplus_{p+q=w} H^{p,q}, \quad 
	H^{p,q} \simeq \mathcal{F}^p / \mathcal{F}^{p+1}, \quad 
	\theta(H^{p,q}) \subset H^{p-1,q+1} \otimes \Omega_{C \setminus D}.
\end{align*}
Let $D_1 = \sum_{i=1}^d a_i \{x_i\}$ and $D_2 = \sum_{i=1}^d b_i \{x_i\}$ be $\mathbb{R}$-divisors supported on $D$. We write $D_1 < D_2$ (resp. $D_1 \leq D_2$) if $a_i < b_i$ (resp. $a_i \leq b_i$) for all $i = 1,\dots,d$.
\begin{defn}\label{defn_prolongation}
	Let $A = \sum_{i=1}^d a_i \{x_i\}$ be an $\mathbb{R}$-divisor on $C$, and let $U \subset C$ be a nonempty open subset. For any holomorphic section $s \in \Gamma(U \setminus D,\, H)$, we write $(s) \leq -A$ if, for every $i = 1,\dots,d$ and for any local holomorphic coordinate $z_i$ centered at $x_i$ (i.e., $x_i = \{z_i = 0\}$), the pointwise norm $|s|_{h_Q}$ satisfies  
	\[
	|s|_{h_Q} = O\big(|z_i|^{-a_i - \varepsilon}\big) \quad \text{as } z_i \to 0,
	\]  
	for all $\varepsilon > 0$.  
	
	The $\mathcal{O}_C$-module ${}_A H$ is then defined by specifying, for each open $U \subset C$, its sections as  
	\[
	\Gamma(U,\, {}_A H) := \big\{ s \in \Gamma(U \setminus D,\, H) \mid (s) \leq -A \big\}.
	\]
\end{defn}
Let
\begin{align*}
	{_{<A}}H:=\bigcup_{{B}<{A}}{_{B}}H\quad\textrm{and}\quad{\rm Gr}_{A}H:={_{A}}H/{_{<A}}H.
\end{align*}	
According to Simpson \cite[Theorem 3, Theorem 5]{Simpson1990}, the set of prolongations forms a parabolic structure.
\begin{thm}\label{thm_parabolic}
	For each $\bR$-divisor $A$ supported on $D$, $_{A}H$ is a locally free coherent sheaf satisfying the following conditions:
	\begin{itemize}
		\item  $_{A+\epsilon D_i}H = {_A}H$ for any $i=1,\dots,l$ and any constant $0<\epsilon\ll 1$,
		\item $_{A+D_i}H={_A}H\otimes \sO(-D_i)$ for every $1\leq i\leq l$,
		\item the subset of $(a_1,\dots,a_l)\in\bR^l$ such that ${\rm Gr}_{\sum_{i=1}^l a_iD_i}H\neq 0$ is discrete, and
		\item the Higgs field $\theta$ has at most logarithmic poles along $D$, that is, $\theta$ extends to 
		\begin{align*}
			{_{A}}H\to{_{A}}H\otimes\Omega_{C}(\log D).
		\end{align*}
	\end{itemize} 
    Moreover, the parabolic Higgs bundle $(H,\theta,\{{_A}H\})$ is semistable with trivial parabolic degree.
\end{thm}
\subsection{Hodge theoretic Arakelov inequality}\label{section_Hodge_Arakelov}
Let $S$ be a smooth projective variety, and let $D \subset E \subset S$ be closed algebraic subsets. Let $\mathbb{V} = (\mathbb{V}_\mathbb{R}, V, \nabla, W_\bullet, F^\bullet)$ be an $\mathbb{R}$-polarized graded variation of mixed Hodge structure on $S \setminus E$, admissible with respect to the compactification $(S, E)$. Denote by $w := \max\{p \mid F^p \neq 0\}$ the maximal Hodge index, and define $l(\mathbb{V}) := \min\{p \mid F^{w-p} = V\}$.

We impose the following assumption.
\begin{assumption}\label{assumption_star}
	There exists a closed algebraic subset $Z \subset S$ satisfying $\mathrm{codim}_S(Z) \geq 2$, such that:  
	\begin{enumerate}
		\item[(i)] Both $D \setminus Z$ and $E \setminus Z$ are reduced simple normal crossing divisors on $S \setminus Z$;  
		\item[(ii)] The lower canonical logarithmic Higgs bundle $(\widetilde{H} = \bigoplus_{p \in \mathbb{Z}} \widetilde{H}^p,\, \theta)$ associated with $\mathbb{V}$ (Definition \ref{defn_LHB_ass_VHSBM}) is defined on the log pair $(S \setminus Z,\, E \setminus Z)$;  
		\item[(iii)] There exists a torsion-free rank-one coherent sheaf $L$ on $S$ and an injective morphism $L|_{S \setminus Z} \hookrightarrow \widetilde{H}^w$.  
	\end{enumerate}  
	Moreover, the Higgs subsheaf of $(\widetilde{H}, \theta)$ generated by $L^0 := L|_{S \setminus Z}$ satisfies  
	\[
	\theta(L^p) \subset L^{p+1} \otimes_{\mathcal{O}_{S \setminus Z}} \Omega_{S \setminus Z}(\log (D \setminus Z)), \quad \forall p \geq 0,
	\]  
	where $L^p \subset \widetilde{H}^{w-p}$ for all $p \geq 0$.
\end{assumption}
The main result of this section is the following Arakelov-type inequality.
\begin{thm}\label{thm_Hodge_Arakelov}
	Let $d = \dim S$. Assume Assumption \ref{assumption_star} holds. If $K_S + D$ is pseudo-effective, then for every movable curve class $\alpha \in N_1(S)$ and every integer $n \geq l(\mathbb{V})$, the inequality  
	\begin{align}\label{align_Araineq_abs1}
		c_1(L) \cdot \alpha \leq \frac{1}{n + 1}\sum_{k=1}^{n} k d^{k-1} \, (K_S + D) \cdot \alpha
	\end{align}  
	holds.
	If moreover $(S,D)$ is log smooth and $K_S+D$ is ample, then the inequality  
	\begin{align*}
		c_1(L) \cdot(K_S+D)^{d-1} \leq \frac{n}{d}(K_S+D)^{d}
	\end{align*}  
	holds.
\end{thm}
\begin{proof}
	We first treat the case when $K_S+D$ is pseudo-effective.
	
	Step 1: Reduction to the simple normal crossing setting and integral ample intersection.  
	We reduce the proof to the case where $D$ and $E$ are reduced simple normal crossing divisors and the movable class $\alpha \in N_1(S)$ is represented by the intersection of $d-1$ ample Cartier divisor classes.
	
	Since $\alpha$ is movable, there exists a projective birational morphism $\mu \colon S' \to S$, with $S'$ a smooth projective variety, and ample $\mathbb{R}$-divisor classes $\alpha_1, \dots, \alpha_{d-1} \in N^1(S')_\mathbb{R}$, such that  
	\[
	\alpha = \mu_*(\alpha_1 \cdots \alpha_{d-1}).
	\]  
	By resolving the non-snc loci of $D$ and $E$ via successive blow-ups along smooth centers, we may assume that $\mu^{-1}(D)$ and $\mu^{-1}(E)$ are simple normal crossing divisors on $S'$; in particular, the exceptional set $Z$ in Assumption \ref{assumption_star} becomes empty. The pullback $\mu^*\mathbb{V}$ remains admissible with respect to $(S', \mu^{-1}(E))$, and there exists an ideal sheaf $I \subset \mathcal{O}_{S'}$, supported in codimension at least two outside the $\mu$-exceptional locus, such that the natural map  
	\[
	\mu^*(L^{\vee\vee}) \otimes I \hookrightarrow \widetilde{\mu^* H}^{\,w}
	\]  
	is injective. Consequently, the tuple $(S', \mu^{-1}(D), \mu^{-1}(E), \mu^*\mathbb{V}, \mu^*(L^{\vee\vee}) \otimes I)$ satisfies Assumption \ref{assumption_star}.
	
	Moreover, by the projection formula, we have  
	\[
	c_1(L) \cdot \alpha = c_1\big(\mu^*(L^{\vee\vee}) \otimes I\big) \cdot \alpha_1 \cdots \alpha_{d-1}
	\]  
	and  
	\[
	(K_S + D) \cdot \alpha = (K_{S'} + \mu^{-1}(D)) \cdot \alpha_1 \cdots \alpha_{d-1}.
	\]  
	Hence, it suffices to prove the inequality on $S'$.
	
	Therefore, without loss of generality, we may assume:  
	(i) $D$ and $E$ are reduced simple normal crossing divisors (so the subset $Z$ in Assumption \ref{assumption_star} is empty);  
	(ii) $\alpha = \alpha_1 \cdots \alpha_{d-1}$, where each $\alpha_i \in N^1(S)_\mathbb{R}$ is the class of an ample $\mathbb{R}$-divisor;  
	(iii) by linearity of both sides of (\ref{align_Araineq_abs1}) in $\alpha$, we may further assume each $\alpha_i = [A_i]$ for some ample Cartier divisors $A_1, \dots, A_{d-1}$ on $S$.
    
    \emph{Step 2:} 
	Let $w_0$ be the largest integer such that $L \subset W_{w_0}(\widetilde{H}^w)$ but $L \nsubseteq W_{w_0-1}(\widetilde{H}^w)$. Then the natural inclusion $L \hookrightarrow \mathrm{Gr}^W_{w_0}(\widetilde{H}^w)$ is nonzero. Denote by $\big(\bigoplus_{p \geq 0} L^p_{w_0},\, \theta\big)$ the logarithmic Higgs subsheaf of $\big(\mathrm{Gr}^W_{w_0}(\widetilde{H}),\, \theta\big)$ generated by $L^0_{w_0} := L$, where $L^p_{w_0} \subset \mathrm{Gr}^W_{w_0}\big(\widetilde{H}^{w-p}\big)$ for each $p \geq 0$. The logarithmic Higgs field  
	\[
	\theta \colon L^p_{w_0} \to L^{p+1}_{w_0} \otimes_{\mathcal{O}_S} \Omega_S(\log E)
	\]  
	is holomorphic on $S \setminus D$ for all $p \geq 0$.
	
	Consider the induced complex:  
	\begin{align*}  
		L = L^0_{w_0} \xrightarrow{\,\theta\,} L^1_{w_0} \otimes_{\mathcal{O}_S} \Omega_S(\log D) \xrightarrow{\,\theta \otimes \mathrm{id}\,} L^2_{w_0} \otimes_{\mathcal{O}_S} \Omega_S^{\otimes 2}(\log D) \to \cdots.
	\end{align*}  
	Let $m_0 := \max\{p \mid L^p_{w_0} \neq 0\}$. Since $\bigoplus_{p=0}^{m_0} L^p_{w_0}$ is generated by $L$, there exist natural surjective morphisms  
	\begin{align*}
		L \otimes_{\mathcal{O}_S} T_S(-\log D)^{\otimes p} \twoheadrightarrow L^p_{w_0}, \quad p = 0, \dots, m_0.
	\end{align*}  
	Dualizing and taking reflexive hulls yields injective morphisms in codimension one:  
	\begin{align}\label{align_gammap}
		\gamma^p \colon L^{\vee\vee} \otimes_{\mathcal{O}_S} (L^p_{w_0})^\vee \hookrightarrow \Omega_S^{\otimes p}(\log D), \quad p = 0, \dots, m_0.
	\end{align}  	
	Since $\omega_S(D)$ is pseudo-effective, it follows from \cite[Theorem 1.2]{CP2019} that the first Chern class of the quotient sheaf $\Omega_S^{\otimes p}(\log D)/\operatorname{Im}(\gamma^p)$ is pseudo-effective for each $p = 1, \dots, m_0$. Consequently, for all such $p$, we have  
	\begin{align*}
		\big(p d^{p-1} c_1(K_S + D) - \operatorname{rank}(L^p_{w_0})\, c_1(L) + c_1(L^p_{w_0})\big) \cdot \alpha \geq 0.
	\end{align*}  
	Summing over $p = 1$ to $m_0$, and noting that $L^0_{w_0} \simeq L$, yields  
	\begin{align}\label{align_inequality1}
		\left( \sum_{p=1}^{m_0} p d^{p-1}\, c_1(K_S + D) - \operatorname{rank}(\mathcal{L})\, c_1(L) \right) \cdot \alpha \geq -c_1(\mathcal{L}) \cdot \alpha,
	\end{align}
	where $\mathcal{L} := \bigoplus_{p=0}^{m_0} L^p_{w_0}$.
	
	Since $\mathcal{L}$ is torsion-free, there exists a dense Zariski open subset $U \subset S$, with $\operatorname{codim}_S(S \setminus U) \geq 2$, on which $\mathcal{L}$ is locally free. Without loss of generality, we may assume the ample Cartier divisors $A_1, \dots, A_{d-1}$ are chosen in general position so that their complete intersection $C := A_1 \cap \cdots \cap A_{d-1}$ is a smooth, connected curve satisfying:  
	(i) $C \subset U$;  
	(ii) $C$ intersects each irreducible component of $E$ transversally.  
	It follows that $\mathcal{L}|_C$ is a logarithmic Higgs subsheaf of $\widetilde{H}_{w_0}|_C$, and moreover,  
	\[
	\widetilde{H}_{w_0}|_C \simeq \widetilde{H_{w_0}|_{C \setminus E}},
	\]  
	where the right-hand side denotes the lower canonical logarithmic Higgs bundle associated with the polarized variation of pure Hodge structure $\mathrm{Gr}^W_{w_0} \mathbb{V}|_{C \setminus E}$.
	
	Let $C \cap E = \{x_1, \dots, x_l\}$. Associated with this variation is a naturally defined parabolic Higgs bundle  
	\[
	\big(\widetilde{H_{w_0}|_{C \setminus E}},\, \theta,\, \{{}_A H_C \mid A = \textstyle\sum_{i=1}^l a_i x_i\}\big),
	\]  
	as described in \S\ref{section_nonabelian_Hodge}. By Theorem \ref{thm_parabolic}, this parabolic Higgs bundle is semistable and has parabolic degree zero; furthermore, the zero-parabolic-weight component satisfies ${}_0 H_C \simeq \widetilde{H}_{w_0}|_C$. Since $\mathcal{L}|_C \subset {}_0 H_C$, it inherits the trivial parabolic structure and thus forms a Higgs subsheaf of $(\widetilde{H_{w_0}|_{C \setminus E}}, \theta)$. Semistability and vanishing parabolic degree therefore imply  
	\begin{align}\label{align_inequality2}
		c_1(\mathcal{L}) \cdot \alpha = c_1(\mathcal{L}|_C) \leq 0.
	\end{align}
	
	Observe that $m_0 \leq l(\mathbb{V}) \leq n$, and since $\mathcal{L}$ is torsion-free and each $L^p_{w_0}$ is nonzero for $p = 0, \dots, m_0$, we have $\operatorname{rank}(\mathcal{L}) \geq m_0 + 1$. Moreover, $c_1(K_S + D) \cdot \alpha \geq 0$ by pseudo-effectivity of $\omega_S(D)$ and movability of $\alpha$.
	Combining inequalities (\ref{align_inequality2}) and (\ref{align_inequality1}), and using the elementary inequality  
	\[
	\frac{\sum_{k=1}^{m_0} k d^{k-1}}{\operatorname{rank}(\mathcal{L})} \leq \frac{\sum_{k=1}^{m_0} k d^{k-1}}{m_0 + 1} \leq \frac{\sum_{k=1}^{n} k d^{k-1}}{n + 1},
	\]  
	we obtain  
	\begin{align*}
		\left( \frac{\sum_{k=1}^{n} k d^{k-1}}{n + 1}\, c_1(K_S + D) - c_1(L) \right) \cdot \alpha 
		&\geq \left( \frac{\sum_{k=1}^{m_0} k d^{k-1}}{\operatorname{rank}(\mathcal{L})}\, c_1(K_S + D) - c_1(L) \right) \cdot \alpha \\
		&\geq -\frac{c_1(\mathcal{L}) \cdot \alpha}{\operatorname{rank}(\mathcal{L})} \geq 0.
	\end{align*}  
	This concludes the proof of Theorem \ref{thm_Hodge_Arakelov} under the assumption that $K_S + D$ is pseudo-effective.
	
	The case in which $K_S + D$ is ample follows analogously. In this setting, $S \setminus D$ admits a K\"ahler-Einstein metric with Poincar\'e-type growth along $D$. This induces a Hermitian-Einstein metric on $\Omega_S(\log D)$, thereby ensuring that $\Omega_S(\log D)^{\otimes k}$ is semistable with respect to $K_S + D$ for all integers $k \geq 1$. Applying this stability property to (\ref{align_gammap}), we obtain the inequalities  
	\begin{align*}  
		\mu(L) - \mu(L_{w_0}^p) \leq p\,\mu(\Omega_S(\log D)), \quad p = 0, \dots, m_0,  
	\end{align*}  
	where $\mu(F)$ denotes the slope of a torsion free coherent sheaf $F$ with respect to $K_S + D$:
	\[
	\mu(F) = \frac{c_1(F) \cdot (K_S + D)^{d-1}}{\operatorname{rank}(F)}.  
	\]  
	
	Summing both sides of the inequality over $p = 1$ to $m_0$ yields  
	\begin{align*}  
		\left( \sum_{p=0}^{m_0} \frac{p}{d}\operatorname{rank}(L_{w_0}^p)\,(K_S + D) - {\rm rank}(\mathcal{L})c_1(L) \right) \cdot (K_S + D)^{d-1} \geq -c_1(\mathcal{L}) \cdot (K_S + D)^{d-1} \geq 0.  
	\end{align*}  
	Since $$\sum_{p=1}^{m_0} p\operatorname{rank}(L_{w_0}^p)\leq m_0{\rm rank}(\mathcal{L})\leq n{\rm rank}(\mathcal{L}),$$ it follows that  
	\begin{align*}  
		\left( \frac{n}{d}(K_S + D) - c_1(L) \right) \cdot (K_S + D)^{d-1} \geq 0.  
	\end{align*}  
	This concludes the proof of Theorem \ref{thm_Hodge_Arakelov} under the assumption that $K_S + D$ is ample.
\end{proof}
\subsection{Higgs sheaf associated with a family}
Motivated by the foundational works of Viehweg-Zuo \cite{VZ2001,VZ2003}, the author of \cite{stjc2025} constructs a Higgs sheaf satisfying Assumption \ref{assumption_star} from a family of simple normal crossing pairs. This construction yields a logarithmic Higgs subsheaf on the base whose positivity properties imply hyperbolicity for the admissible substack of $\overline{\mathcal{M}}_{\mathrm{slc}}(d, \Phi_c, \Gamma, \sigma)$. It also plays a central role in our proof of Theorem \ref{thm_main_numbound_moduli}.
\begin{defn}[Simple normal crossing family]\label{defn_SNC_family}
	Let $S$ and $X$ be reduced schemes of finite type over $\operatorname{Spec}(\mathbb{C})$, and let $D \geq 0$ be a Weil $\bQ$-divisor on $X$. Let $f: X \to S$ be a morphism. The morphism $f: (X, D) \to S$ is said to be \emph{simple normal crossing over $S$} at a point $x \in X$ if there exists a Zariski open neighborhood $U$ of $x$ in $X$ that can be embedded into a scheme $Y$, which is smooth over $S$. In this embedding, $Y$ admits a regular system of parameters $(z_1, \dots, z_p, y_1, \dots, y_r)$ over $S$ at the point corresponding to $x = 0$, such that $U$ is defined by the monomial equation $z_1 \cdots z_p = 0$ and  
	\[ D|_U = \sum_{i=1}^r a_i (y_i = 0)|_U, \quad \text{where } a_i \geq 0, \]  
	over $S$.
	
	The morphism $f: (X, D) \to S$ is said to be a \emph{simple normal crossing family over $S$} if it satisfies the simple normal crossing condition over $S$ at every point of $X$. 
	
	In the special case where $S = \operatorname{Spec}(\mathbb{C})$, the pair $(X, D)$ is referred to as a \emph{simple normal crossing pair} if $(X, D)$ forms a simple normal crossing family over $\operatorname{Spec}(\mathbb{C})$. In this context, $X$ has Gorenstein singularities and possesses an invertible dualizing sheaf $\omega_X$. The canonical divisor $K_X$ is defined up to linear equivalence via the isomorphism $\omega_X \simeq \sO_X(K_X)$.
\end{defn}
We fix a projective morphism $f \colon X \to S$ of quasi-projective reduced schemes over $\mathbb{C}$ of relative dimension $n = \dim X - \dim S$, where $S$ is smooth and $X$ is a simple normal crossing scheme. Let $\Delta$ be a $\mathbb{Q}$-divisor on $X$ with coefficients in $[0,1]$, such that $(X,\Delta)$ is a simple normal crossing pair. Assume there exists a reduced simple normal crossing divisor $D \subset S$ satisfying:  
(i) the restriction $f^o := f|_{X^o} \colon (X^o, \Delta^o) \to S^o$ is a simple normal crossing family, where $S^o := S \setminus D$, $X^o := f^{-1}(S^o)$, and $\Delta^o := \Delta|_{X^o}$;  
(ii) $\mathrm{supp}(\Delta)$ contains no irreducible component of $f^{-1}(D)$.  

Let $L$ be a torsion free sheaf of rank 1 on $S$, and suppose there exists a nonzero morphism  
\begin{align*}
	s_L \colon L^{\otimes k} \to f_*\big(\mathcal{O}_X(kK_{X/S} + k\Delta)\big)  
\end{align*}  
for some integer $k \geq 1$ such that $k\Delta$ is integral.
\begin{thm}[\cite{stjc2025}, Theorem 4.1]\label{thm_VZ_construction}  
	Let the notation be as above. Then there exist:  
	\begin{enumerate}  
		\item[(i)] a closed algebraic subset $E \subset S$ containing $D$, and a closed algebraic subset $Z \subset S$ satisfying $\operatorname{codim}_S(Z) \geq 2$, such that both $D \setminus Z$ and $E \setminus Z$ are reduced simple normal crossing divisors on $S \setminus Z$;  
		\item[(ii)] an $\mathbb{R}$-polarized variation of mixed Hodge structure $\mathbb{V}$ on $S \setminus E$, admissible with respect to the compactification $(S, E)$, with $l(\mathbb{V}) \leq n$;  
		\item[(iii)] the lower canonical logarithmic Higgs bundle $(\widetilde{H} = \bigoplus_{p \in \mathbb{Z}} \widetilde{H}^p,\, \theta)$, defined on the log pair $(S \setminus Z,\, E \setminus Z)$ and associated with $\mathbb{V}$;  
		\item[(iv)] an injective morphism  
		\[
		L(-D)|_{S \setminus Z} \hookrightarrow \widetilde{H}^w,
		\]  
		where $L(-D) := L \otimes_{\mathcal{O}_S} \mathcal{O}_S(-D)$.  
	\end{enumerate}  
	Moreover, the Higgs subsheaf $\bigoplus_{p \geq 0} L^p \subset \bigoplus_{p \geq 0} \widetilde{H}^{w-p}$ generated by $L^0 := L(-D)|_{S \setminus Z}$ has logarithmic poles along $D$:  
	\[
	\theta(L^p) \subset L^{p+1} \otimes_{\mathcal{O}_{S \setminus Z}} \Omega_{S \setminus Z}^1(\log(D \setminus Z)), \quad \forall\, p \geq 0.
	\]
\end{thm}
Combined with Theorem \ref{thm_Hodge_Arakelov}, we obtain:
\begin{thm}\label{thm_VZ_Arakelov}  
	Let the notation be as above. Let $d = \dim S$ and $n = \dim X - \dim S$. If $K_S + D$ is pseudo-effective, then the following Arakelov-type inequality holds:  
	\begin{align*}
		(c_1(L)-D) \cdot \alpha \leq \frac{1}{n + 1}\sum_{p=1}^{n}pd^{p-1}(K_S + D) \cdot \alpha
	\end{align*}  
	for every movable curve class $\alpha \in N_1(S)$.	
	If moreover $K_S+D$ is ample, then the inequality  
	\begin{align*}
		(c_1(L)-D) \cdot(K_S+D)^{d-1} \leq \frac{n}{d}(K_S+D)^{d}
	\end{align*}  
	holds.
\end{thm}
	\section{Boundedness of polarized algebraic families}\label{section_boundedness}
	\subsection{Stable minimal models and their moduli}\label{section_moduli}
	We now review the main results from \cite{Birkar2022} that will be utilized in the subsequent sections. A \emph{stable minimal model} is a triple $(X, B), A$, where $X$ is a reduced, connected, projective scheme of finite type over ${\rm Spec}(\bC)$, and $A, B \geq 0$ are $\bQ$-divisors satisfying the following conditions:  
	\begin{itemize}  
		\item $(X, B)$ is a projective, connected slc (semi-log-canonical) pair,  
		\item $K_X + B$ is semi-ample,  
		\item $K_X + B + tA$ is ample for some $t > 0$, and  
		\item $(X, B + tA)$ is slc for some $t > 0$.  
	\end{itemize}  
	Let  
	$$
	d \in \bN, \, c \in \bQ^{\geq 0}, \, \Gamma \subset \bQ^{>0} \text{ a finite set, and } \sigma \in \bQ[t].
	$$  
	A $(d, \Phi_c, \Gamma, \sigma)$-stable minimal model is a stable minimal model $(X, B), A$ satisfying the following conditions:  
	\begin{itemize}  
		\item $\dim X = d$,  
		\item the coefficients of $A$ and $B$ belong to $c \bZ^{\geq 0}$,  
		\item ${\rm vol}(A|_F) \in \Gamma$, where $F$ is any general fiber of the fibration $f: X \to Z$ determined by $K_X + B$, and  
		\item ${\rm vol}(K_X + B + tA) = \sigma(t)$ for $0 \leq t \ll 1$.  
	\end{itemize}  
	Let $S$ be a reduced scheme over ${\rm Spec}(\bC)$. A family of $(d, \Phi_c, \Gamma, \sigma)$-stable minimal models over $S$ consists of a projective morphism $X \to S$ of schemes and $\bQ$-divisors $A$ and $B$ on $X$, satisfying the following conditions:  
	\begin{itemize}  
		\item $(X, B + tA) \to S$ is a locally stable family (i.e., $K_{X/S} + B + tA$ is $\bQ$-Cartier) for every sufficiently small rational number $t \geq 0$,  
		\item $A = cN$, $B = cD$, where $N, D \geq 0$ are relative Mumford divisors, and  
		\item $(X_s, B_s), A_s$ is a $(d, \Phi_c, \Gamma, \sigma)$-stable minimal model for each point $s \in S$.  
	\end{itemize} 
	Let ${\rm Sch}_{\bC}^{\rm red}$ denote the category of reduced schemes defined over ${\rm Spec}(\bC)$. Define  
	$$
	\sM^{\rm red}_{\rm slc}(d, \Phi_c, \Gamma, \sigma): S \mapsto \{\text{families of } (d, \Phi_c, \Gamma, \sigma)\text{-stable minimal models over } S\},
	$$  
	a functor of groupoids over ${\rm Sch}_{\bC}^{\rm red}$.  
	
	\begin{thm}[Birkar \cite{Birkar2022}] \label{thm_moduli_stable_var}  
		There exists a proper Deligne-Mumford stack $\sM_{\rm slc}(d, \Phi_c, \Gamma, \sigma)$ over ${\rm Spec}(\bC)$ such that the following properties hold:  
		\begin{itemize}  
			\item $\sM_{\rm slc}(d, \Phi_c, \Gamma, \sigma)|_{{\rm Sch}_{\bC}^{\rm red}} = \sM^{\rm red}_{\rm slc}(d, \Phi_c, \Gamma, \sigma)$ as functors of groupoids.  
			\item $\sM_{\rm slc}(d, \Phi_c, \Gamma, \sigma)$ admits a projective good coarse moduli space $M_{\rm slc}(d, \Phi_c, \Gamma, \sigma)$.  
		\end{itemize}  
	\end{thm}  
	\begin{proof}  
		See the proof of \cite[Theorem 1.14]{Birkar2022}. Using the notations in \cite[\S 10.7]{Birkar2022}, we have  
		$$
		\sM_{\rm slc}(d, \Phi_c, \Gamma, \sigma) = \left[M_{\rm slc}^e(d, \Phi_c, \Gamma, \sigma, a, r, \bP^n)/{\rm PGL}_{n+1}(\bC)\right],
		$$  
		where the right-hand side denotes the stacky quotient.  
	\end{proof}
	\subsection{Polarization on $M_{\rm slc}(d,\Phi_c,\Gamma,\sigma)$}\label{section_polarization_moduli}
	In this section, we consider certain natural ample $\bQ$-line bundles on $M_{\rm slc}(d, \Phi_c, \Gamma, \sigma)$. Their constructions are implicitly described in the proof of \cite[Theorem 1.14]{Birkar2022}, relying on Koll\'ar's ampleness criterion \cite{Kollar1990}.  
	
	Fix the data $d, \Phi_c, \Gamma, \sigma$. Since $\sM_{\rm slc}(d, \Phi_c, \Gamma, \sigma)$ is of finite type, there exist constants  
	$$
	(a, r, j) \in \bQ^{\geq 0} \times (\bZ^{>0})^2,
	$$  
	depending only on $d, \Phi_c, \Gamma, \sigma$, such that every $(d, \Phi_c, \Gamma, \sigma)$-stable minimal model $(X, B), A$ satisfies the following conditions (cf. \cite[Lemma 10.2]{Birkar2022}):  
	\begin{itemize}  
		\item $(X, B + aA)$ is an slc pair,  
		\item $r(K_X + B + aA)$ is a very ample integral Cartier divisor with  
		$$
		H^i(X, \sO_X(k r(K_X + B + aA))) = 0, \quad \forall i > 0, \forall k > 0,  
		$$  
		\item the embedding $X \hookrightarrow \bP(H^0(X, r(K_X + B + aA)))$ is defined by equations of degree $\leq j$, and  
		\item the multiplication map  
		$$
		S^j(H^0(X, \sO_X(r(K_X + B + aA)))) \to H^0(X, \sO_X(j r(K_X + B + aA)))
		$$  
		is surjective.  
	\end{itemize}
	\begin{defn}
		A tuple $(a, r, j) \in \bQ^{\geq 0} \times (\bZ^{>0})^2$ that satisfies the conditions above is referred to as a \emph{$(d, \Phi_c, \Gamma, \sigma)$-polarization datum}.
	\end{defn}
	Let $(a, r, j) \in \bQ^{\geq 0} \times (\bZ^{>0})^2$ be a $(d, \Phi_c, \Gamma, \sigma)$-polarization datum. Let $(X, B), A \to S$ be a family of $(d, \Phi_c, \Gamma, \sigma)$-stable minimal models. Then $f_\ast(r(K_{X/S} + B + aA))$ is locally free and commutes with arbitrary base changes. Therefore, the assignment  
	$$
	f: (X, B), A \to S \in \sM_{\rm slc}(d, \Phi_c, \Gamma, \sigma)(S) \mapsto f_\ast(r(K_{X/S} + B + aA))
	$$  
	defines a locally free coherent sheaf on the stack $\sM_{\rm slc}(d, \Phi_c, \Gamma, \sigma)$, denoted by $\Lambda_{a,r}$. Let $\lambda_{a,r} := \det(\Lambda_{a,r})$. Since $\sM_{\rm slc}(d, \Phi_c, \Gamma, \sigma)$ is Deligne-Mumford, some power $\lambda_{a,r}^{\otimes k}$ descends to a line bundle on $M_{\rm slc}(d, \Phi_c, \Gamma, \sigma)$. For this reason, we regard $\lambda_{a,r}$ as a $\bQ$-line bundle on $M_{\rm slc}(d, \Phi_c, \Gamma, \sigma)$.
	\begin{prop}\label{prop_ample_line_bundle_moduli}  
		Let $(a, r, j) \in \bQ^{\geq 0} \times (\bZ^{>0})^2$ be a $(d, \Phi_c, \Gamma, \sigma)$-polarization datum. Then $\lambda_{a,r}$ is ample on $M_{\rm slc}(d, \Phi_c, \Gamma, \sigma)$.  
	\end{prop}  	
	\begin{proof}  
		By the same arguments as in \cite[\S 2.9]{Kollar1990}. It suffices to show that $f_\ast(r(K_{X/S} + B + aA))$ is nef when $S$ is a smooth projective curve. This was established by Fujino \cite{Fujino2018} and Kov\'acs-Patakfalvi \cite{Kovacs2017}.  
	\end{proof}  
	\subsection{Fiber product of stable families}
	Following Koll\'ar \cite{Kollar2023}, a flat family of slc pairs $(X, \Delta) \to S$ over a smooth base $S$ is referred to as a \emph{locally stable family} if $K_{X/S} + \Delta$ is $\bQ$-Cartier.
	
	Let $S$ be a smooth variety, and let $f:(X, \Delta) \to S$ be a locally stable family of slc pairs. Denote by $X^{[r]}_S$ the $r$-fold fiber product $X \times_S X \times_S \cdots \times_S X$, and let $f^{[r]}: X^{[r]}_S \to S$ be the projection map. Define  
	$$
	\Delta^{[r]}_S := \sum_{i=1}^r p_i^\ast(\Delta),
	$$  
	where $p_i: X^{[r]}_S \to X$ is the projection onto the $i$-th component. Then $f^{[r]}:(X^{[r]}_S, \Delta^{[r]}_S) \to S$ is also a locally stable family of slc pairs (\cite[Corollary 4.3]{WeiWu2023}). According to \cite[Theorem 4.54]{Kollar2023}, $(X^{[r]}_S, \Delta^{[r]}_S)$ is an slc pair.
	
	Let $\tau: (X^{(r)}_S, \Delta^{(r)}_S) \to (X^{[r]}_S, \Delta^{[r]}_S)$ be a semi-log resolution (\cite{Bierstone2013} or \cite[\S 10.4]{Kollar2013}), where $\Delta^{(r)}_S$ is a simple normal crossing divisor determined by  
	\begin{align*}
		\tau^\ast(K_{X^{[r]}_S} + \Delta^{[r]}_S) = K_{X^{(r)}_S} + \Delta^{(r)}_S - E,  
	\end{align*}  
	with $E$ being an effective $\tau$-exceptional divisor that shares no common components with $\Delta^{(r)}_S$. 
	\begin{lem}(\cite[Lemma 5.4]{stjc2025})\label{lem_mild_pushforward}
		Let $k \geq 1$ such that $kK_{X/S} + k\Delta$ is integral. Then the following statements hold:  
		\begin{enumerate}
			\item $\tau_{\ast}\left(\sO_{X^{(r)}_S}(kK_{X^{(r)}_S/S}+k\Delta^{(r)}_S)\right)\simeq\sO_{X^{[r]}_S}(kK_{X^{[r]}_S/S}+k\Delta^{[r]}_S)$.
			\item $f^{[r]}_\ast\left(\sO_{X^{[r]}_S}(kK_{X^{[r]}_S/S}+k\Delta^{[r]}_S)\right)$ is a reflexive sheaf.
			\item $f^{[r]}_\ast\left(\sO_{X^{[r]}_S}(kK_{X^{[r]}_S/S}+k\Delta^{[r]}_S)\right)\simeq (f_\ast\left(\sO_X(kK_{X/S}+k\Delta)\right)^{\otimes r})^{\vee\vee}$.
		\end{enumerate}
	\end{lem}
	We will also need the following lemma (originated to Viehweg \cite{Viehweg1983}) in the sequel.
	\begin{lem}(\cite[Lemma 5.5]{stjc2025})\label{lem_Viehweg_inclusion}
		Let $g: Y \to S$ be a projective surjective morphism from a simple normal crossing variety $Y$ to a smooth variety $S$. Let $\tau: S' \to S$ be a flat projective surjective morphism from a smooth variety $S'$. Let $\Delta \geq 0$ be an integral Cartier divisor on $Y$. Consider the following commutative diagram:  
		\begin{align*}  
			\xymatrix{  
				Y \ar[d]^g & Y'' \ar[l]_{\rho'} \ar[d]^{g''} & Y' \ar[l]_{\rho} \ar[dl]^{g'} \\  
				S & S' \ar[l]_{\tau} &  
			},  
		\end{align*}  
		where $g'': Y'' \to S'$ is the base change of $g$, and $\rho: Y' \to Y''$ is a semi-log resolution of singularities. Then, for every $r \geq 1$, there exists a natural inclusion  
		$$  
		g'_\ast\left(\sO_{Y'}(rK_{Y'/S'} + (\rho'\rho)^\ast\Delta)\right) \subset \tau^\ast g_\ast\left(\sO_Y(rK_{Y/S} + \Delta)\right).  
		$$  
	\end{lem}
	\subsection{Proof of Theorem \ref{thm_main_numbound_moduli}}\label{section_VZHiggs_stable_family}
	The objective of this section is to establish the Arakelov inequality for families of stable minimal models, as stated in Theorem \ref{thm_Arakelov_family}.
	\begin{defn}\label{defn_semipair}
		A \emph{semi-pair} $(X, \Delta)$ consists of a reduced connected scheme of finite type over ${\rm Spec}(\bC)$, of pure dimension, together with a $\bQ$-divisor $\Delta\geq 0$ on $X$, satisfying the following conditions:  
		\begin{enumerate}  
			\item $X$ is an $S_2$-scheme with nodal singularities in codimension one,  
			\item no component of ${\rm Supp}(\Delta)$ is contained in the singular locus of $X$,  
			\item $K_X + \Delta$ is $\bQ$-Cartier.  
		\end{enumerate}  
		A \emph{projective semi-pair} is a semi-pair $(X, \Delta)$ where $X$ is a projective scheme.
	\end{defn}
	\begin{defn}\label{defn_semiresolution_semipair}
		Let $(X, \Delta)$ be a projective semi-pair. A morphism between projective semi-pairs $\pi: (X', \Delta') \to (X, \Delta)$ is said to be a \textit{semi-log resolution} if $\pi: X' \to X$ is a semi-log resolution of singularities (\cite{Bierstone2013} or \cite[\S 10.4]{Kollar2013}) such that there exists a $\pi$-exceptional $\bQ$-divisor $E \geq 0$ on $X'$ satisfying the relation $K_{X'} + \Delta' = \pi^{\ast}(K_X + \Delta) + E$.
	\end{defn}
	To provide a precise definition of a birationally admissible family, we extend Fujino's concept of B-birational maps \cite[Definition 1.5]{Fujino2000}, which naturally generalizes birational maps to the context of log pairs.
	\begin{defn}
		Let $(X_1, \Delta_1)$ and $(X_2, \Delta_2)$ be two projective semi-pairs over a variety $S$. A rational map $f: (X_1, \Delta_1) \dasharrow (X_2, \Delta_2)$ over $S$ is said to be a \textit{log birational map over $S$} if there exists a common semi-log resolution $\pi_i: (X', \Delta') \to (X_i, \Delta_i)$, $i = 1, 2$, over $S$ such that $f = \pi_2 \circ \pi_1^{-1}$. We say that a morphism $f: (X, \Delta) \to S$ admits a \textit{simple normal crossing log birational model} if there exists a simple normal crossing family $f': (X', \Delta') \to S$ and a log birational map $(X, \Delta) \dasharrow (X', \Delta')$ over $S$.
	\end{defn}
	\begin{defn}\label{defn_admissible}
		Let $f: (X, B), A \to S$ be a family of stable minimal models over a variety $S$. The morphism $f$ is said to be \emph{birationally admissible} if $(X, B + A) \to S$ admits a simple normal crossing log birational model.
	\end{defn}
	A prototypical example of a birationally admissible family of stable minimal models is a family $f \colon (X, B), A \to S$ such that the pair $(X, B + A) \to S$ admits a simultaneous semi-log resolution. More precisely, there exists a morphism $(X', \Delta') \to (X, B + A)$ that is a semi-log resolution and satisfies the following conditions: the induced family $(X', \Delta') \to S$ is a simple normal crossing family, and for every $s \in S$, the fiberwise morphism $(X'_s, \Delta'_s) \to (X_s, B_s + A_s)$ is itself a semi-log resolution. By semi-log desingularization procedure (c.f. \cite{Bierstone2013}), any family of stable minimal models becomes birationally admissible after restricting to a dense Zariski open subset of the base.
	\begin{thm}\label{thm_Arakelov_family}
		Let $f^o \colon (X^o, B^o) \to S^o$ be a birationally admissible family of $(d, \Phi_c, \Gamma, \sigma)$-stable minimal models over a smooth quasi-projective $n$-fold $S^o$, inducing a morphism $\xi^o \colon S^o \to M_{\mathrm{slc}}(d, \Phi_c, \Gamma, \sigma)$. Let $S$ be a smooth projective compactification of $S^o$ such that $D := S \setminus S^o$ is a reduced simple normal crossing divisor and $\xi^o$ extends to a morphism $\xi \colon S \to M_{\mathrm{slc}}(d, \Phi_c, \Gamma, \sigma)$. Let $(a,r,j)$ be a $(d, \Phi_c, \Gamma, \sigma)$-polarization datum. Define  
		\[
		\mathrm{ind}_\xi(\lambda_{a,r}) := \min\big\{k \in \mathbb{Z}_{>0} \,\big|\, \xi^*\lambda_{a,r}^{\otimes k} \text{ is a line bundle on } S \big\}.
		\]  
		If $K_S + D$ is pseudo-effective, then for every movable curve class $\alpha \in N_1(S)$ and every positive integer $k$ divisible by $\mathrm{ind}_\xi(\lambda_{a,r})$, the inequality  
		\begin{align}\label{align_Ara_1}
			c_1\big(\xi^*\lambda_{a,r}\big) \cdot \alpha 
			\leq \frac{1}{k(k l r d + 1)} \sum_{p=1}^{k l r d} p n^{p-1} (K_S + D) \cdot \alpha + \frac{2}{k}\, D \cdot \alpha,
		\end{align}  
		holds, where $l = \operatorname{rank}(\Lambda_{a,r})$.
		If moreover $K_S+D$ is ample, then  
		\begin{align}\label{align_Ara_2}
			c_1(\xi^*\lambda_{a,r})\cdot(K_S+D)^{n-1} \leq \frac{lrd}{n}(K_S+D)^{n}.
		\end{align}
	\end{thm}
    \begin{rmk}
    	When $f^o$ extends to a family of stable minimal models over $S$-that is, when the classifying morphism $S^o \to \sM_{\mathrm{slc}}(d, \Phi_c, \Gamma, \sigma)$ extends to a morphism $S \to \sM_{\mathrm{slc}}(d, \Phi_c, \Gamma, \sigma)$-one has $\operatorname{ind}_S(\lambda_{a,r}) = 1$.
    \end{rmk}
	\begin{proof}
		The proof proceeds by applying Theorem \ref{thm_VZ_Arakelov} to establish the desired inequality.
		
		\emph{Step 1: Compactify the family.}
		By the properness of $\sM_{\rm slc}(d, \Phi_c, \Gamma, \sigma)$, we obtain the following constructions:  
		\begin{itemize}  
			\item A generically finite, proper, and surjective morphism $\sigma: \widetilde{S} \to S$ from a smooth projective variety $\widetilde{S}$ such that $\sigma^{-1}(D)$ is a simple normal crossing divisor. The morphism $\sigma$ is constructed as a combination of smooth blow-ups, with centers lying over $S \setminus S^o$, and a finite flat morphism. 
			\item Let $\widetilde{S}^o := \sigma^{-1}(S^o)$ and $\widetilde{X}^o := \widetilde{S}^o \times_{S^o} X^o$. Denote by $\widetilde{A}^o$ and $\widetilde{B}^o$ the divisorial pullbacks of $A^o$ and $B^o$ on $\widetilde{X}^o$, respectively. There exists a completion $\widetilde{f}: (\widetilde{X}, \widetilde{B}), \widetilde{A} \to \widetilde{S}$ of the base change family $(\widetilde{X}^o, \widetilde{B}^o), \widetilde{A}^o \to \widetilde{S}^o$ such that $\widetilde{f} \in \sM_{\rm slc}(d, \Phi_c, \Gamma, \sigma)(\widetilde{S})$.  
		\end{itemize}
		\emph{Step 2: Take the simple normal crossing birational models.} Let $(a, r, j) \in \bQ^{\geq 0} \times (\bZ^{>0})^2$ be a $(d, \Phi_c, \Gamma, \sigma)$-polarization datum. In this case, $(X^o, B^o + aA^o) \to S^o$ is a locally stable family. We denote $\Delta^o := B^o + aA^o$.  
		
		Since $f^o$ is birationally admissible, we may assume that $a$ is sufficiently small such that the following diagrams exist:  
		\begin{align}\label{align_log_smooth_models}
			\xymatrix{
				(X^o, \Delta^o) \ar[rd]_{f^o} & (X^o_1, \Delta^o_1) \ar[l] \ar[r] \ar[d] & (X^o_2, \Delta^o_2) \ar[ld]\\
				& S^o &
			}, \quad
			\xymatrix{
				(\widetilde{X}^o, \widetilde{\Delta}^o) \ar[rd]_{\widetilde{f}^o} & (\widetilde{X}^o_1, \widetilde{\Delta}^o_1) \ar[l] \ar[r] \ar[d] & (\widetilde{X}^o_2, \widetilde{\Delta}^o_2) \ar[ld]\\
				& \widetilde{S}^o &
			},
		\end{align}  
		where $(X^o_1, \Delta^o_1) \to (X^o, \Delta^o)$ and $(X^o_1, \Delta^o_1) \to (X^o_2, \Delta^o_2)$ are semi-log resolutions over $S^o$, and $(X^o_2, \Delta^o_2)$ is a simple normal crossing family over $S^o$ where the coefficients of $\Delta^o_2$ lie in $[0,1]$. The diagram on the right is the base change of the diagram on the left.
		
		\emph{Step 3: Viehweg's fiber product trick.} Now we consider the fiber products of (\ref{align_log_smooth_models}). Let $N > 0$ be an integer. Denote by $X_{1S^o}^{o[N]}$ the main components of the $N$-fold fiber product $X_1^o \times_{S^o} \cdots \times_{S^o} X_1^o$. Let $p_i: X_{1S^o}^{o[N]} \to X_1^o$ be the projection map onto the $i$-th component. We define $\Delta_{1}^{o[N]} := \sum_{i=1}^N p_i^\ast(\Delta_1^o)$ and similarly define $X^{o[N]}_{2S^o}$, $\Delta^{o[N]}_2$, etc.  
		
		Let $\varrho: X_{1S^o}^{o(N)} \to X_{1S^o}^{o[N]}$ be a semi-log resolution of singularities such that ${\rm supp}(\varrho^\ast(\Delta_{1}^{o[N]})) \cup {\rm Exc}(\varrho)$ is a simple normal crossing divisor. Let $E^{+}, E^{-} \geq 0$ be the simple normal crossing $\bQ$-divisors on $X_{1S^o}^{o(N)}$ determined by  
		$$
		\varrho^\ast\left(\sum_{i=1}^N p_i^\ast(K_{X^o_1/S^o} + \Delta_1^o)\right) = K_{X_{1S^o}^{o(N)}/S^o} + E^{+} - E^{-},
		$$  
		where $E^{+}$ and $E^{-}$ have no common components and $E^-$ is $\varrho$-exceptional.  		
		Then the composition morphisms  
		$$
		(X_{1S^o}^{o(N)}, E^+) \to (X^{o[N]}_{S^o}, \Delta^{o[N]}), \quad (X_{1S^o}^{o(N)}, E^+) \to (X^{o[N]}_{2S^o}, \Delta^{o[N]}_2)
		$$  
		are log birational morphisms. Based on the discussions above, we obtain the following commutative diagram:  
		\begin{align}\label{align_fiberprods}
			\xymatrix{
				(\widetilde{X}^{o(N)}_{\widetilde{S}^o}, \widetilde{E}^+) \ar[d]^{\pi^o} \ar[r]^{h^o} & (X_{1S^o}^{o(N)}, E^+) \ar[d]^{\alpha} \ar[r] & (X^{o[N]}_{2S^o}, \Delta^{o[N]}_2) \ar[ddl]\\
				(\widetilde{X}^{o[N]}_{\widetilde{S}^o}, \widetilde{\Delta}^{o[N]}) \ar[r]^{h'} \ar[d] & (X^{o[N]}_{S^o}, \Delta^{o[N]}) \ar[d] &\\
				\widetilde{S}^o \ar[r] & S^o
			},
		\end{align}  
		where $\widetilde{X}^{o(N)}_{\widetilde{S}^o}$ is a semi-log resolution of the main components of $X_{1S^o}^{o(N)} \times_{S^o} \widetilde{S}^o$, such that $\widetilde{E}^+ := h^{o\ast}(E^+)$ is a simple normal crossing divisor. Consequently, 
		\begin{align}\label{align_adjunction_pi0}
			\pi^{o\ast}(K_{\widetilde{X}^{o[N]}_{\widetilde{S}^o}/\widetilde{S}^o}+\widetilde{\Delta}^{o[N]})&=(h'\pi^o)^\ast(K_{X^{o[N]}_{S^o}/S^o}+\Delta^{o[N]})\\\nonumber
			&=(\alpha h^o)^\ast(K_{X^{o[N]}_{S^o}/S^o}+\Delta^{o[N]})\\\nonumber
			&=h^{o\ast}(K_{X^{o(N)}_{1S^o}/S^o}+E^+-E')\quad(E'\textrm{ is }\alpha-\textrm{exceptional})\\\nonumber
			&=K_{\widetilde{X}^{o(N)}_{\widetilde{S}^o}/\widetilde{S}^o}+\widetilde{E}^+-F^o
		\end{align}
		for some divisor $F^o\geq0$. 
		Taking a completion of (\ref{align_fiberprods}), we obtain the following diagram:  
		\begin{align}\label{align_fiberprods_complete}  
			\xymatrix{  
				(\widetilde{Y}^{(N)}, \widetilde{\Delta}^{(N)}) \ar[d]^{\pi} \ar[r]^{h} & (Y^{(N)}_1, \Delta^{(N)}_1) \ar[dd]^{g_1} \ar[r]^{g} & (Y^{(N)}_2, \Delta^{(N)}_2) \ar[ddl]^{g_2}\\  
				(\widetilde{X}^{[N]}_{\widetilde{S}}, \widetilde{\Delta}^{[N]}) \ar[d]^{\widetilde{f}^{[N]}} &&\\  
				\widetilde{S} \ar[r]^{\sigma} & S  
			}  
		\end{align}  
		where the following conditions are satisfied:  
		\begin{itemize}  
			\item $Y^{(N)}_2$ is a simple normal crossing projective scheme such that $Y^{(N)}_2 \to S$ is a completion of $X^{o[N]}_{2S^o} \to S^o$, and $\Delta^{(N)}_2$, the closure of $\Delta^{o[N]}_2$ in $Y^{(N)}_2$, is a simple normal crossing divisor;  
			\item $Y^{(N)}_1$ is a simple normal crossing projective scheme such that $Y^{(N)}_1 \to S$ is a completion of $X^{o(N)}_{1S^o} \to S^o$. The divisor $\Delta^{(N)}_1 \geq 0$ contains $E^+$ and is chosen such that $g$ is a semi-log resolution;  
			\item $\widetilde{Y}^{(N)}$ is a simple normal crossing projective scheme such that $\widetilde{Y}^{(N)} \to \widetilde{S}$ is a completion of $\widetilde{X}^{o(N)}_{\widetilde{S}^o} \to \widetilde{S}^o$, and $(\widetilde{f}^{[N]} \pi)^{\ast}(D) \cup \overline{\widetilde{E}^+}$ is a simple normal crossing divisor. The divisor $\widetilde{\Delta}^{(N)} \geq 0$ is determined by  
			\begin{align}\label{align_logres_XN}  
				\pi^\ast(K_{\widetilde{X}_{\widetilde{S}}^{[N]}} + \widetilde{\Delta}^{[N]}) = K_{\widetilde{Y}^{(N)}} + \widetilde{\Delta}^{(N)} - F,  
			\end{align}  
			where $F \geq 0$ is a $\pi$-exceptional divisor that shares no common components with $\widetilde{\Delta}^{(N)}$.  
		\end{itemize}
		Observe that  
		$$
		\widetilde{\Delta}^{(N)}|_{\widetilde{X}^{o(N)}_{\widetilde{S}^o}} \leq \widetilde{E}^+ = h^{\ast}(\Delta_1^{(N)})|_{\widetilde{X}^{o(N)}_{\widetilde{S}^o}}
		$$  
		according to (\ref{align_adjunction_pi0}). Combining this with the fact that $(\widetilde{X}^{[N]}_{S}, \widetilde{\Delta}^{[N]})$ is an slc pair, we conclude that  
		\begin{align}\label{align_divisors}
			\widetilde{\Delta}^{(N)} \leq h^{\ast}(\Delta_1^{(N)}) + (\sigma\widetilde{f}^{[N]} \pi)^{\ast}(D).
		\end{align}
		According to (\ref{align_logres_XN}), we have  
		\begin{align}\label{align_log_res_directimage}  
			\widetilde{f}^{[N]}_\ast\left(\sO_{\widetilde{X}^{[N]}_{\widetilde{S}}}(rK_{\widetilde{X}^{[N]}_{\widetilde{S}}/\widetilde{S}} + r\widetilde{\Delta}^{[N]})\right) \simeq (\widetilde{f}^{[N]} \pi)_\ast\left(\sO_{\widetilde{Y}^{(N)}}(rK_{\widetilde{Y}^{(N)}/\widetilde{S}} + r\widetilde{\Delta}^{(N)})\right).  
		\end{align}  		
		Since $g$ is a semi-log resolution, there exists an isomorphism  
		\begin{align*}  
			g_{1\ast}\left(\sO_{Y^{(N)}_1}(rK_{Y^{(N)}_1/S} + r\Delta^{(N)}_1)\right) \simeq g_{2\ast}\left(\sO_{Y^{(N)}_2}(rK_{Y^{(N)}_2/S} + r\Delta^{(N)}_2)\right).  
		\end{align*}  		
		By Lemma \ref{lem_Viehweg_inclusion} and (\ref{align_divisors}), we obtain an inclusion  
		\begin{align}\label{align_Vieh_inclusion}  
			(\widetilde{f}^{[N]} \pi)_\ast\left(\sO_{\widetilde{Y}^{(N)}}(rK_{\widetilde{Y}^{(N)}/\widetilde{S}} + r\widetilde{\Delta}^{(N)})\right) \otimes \sigma^\ast(I) \subset \sigma^\ast g_{1\ast}\left(\sO_{Y^{(N)}_1}(rK_{Y^{(N)}_1/S} + r\Delta^{(N)}_1 + rg_1^\ast(D))\right),  
		\end{align}  
		where $I$ is an ideal sheaf on $S$ whose co-support lies in the codimension $\geq 2$ loci over which $\sigma$ is not flat.  		
		Combining (\ref{align_log_res_directimage})--(\ref{align_Vieh_inclusion}), we conclude that  
		\begin{align}\label{align_finally}  
			\widetilde{f}^{[N]}_\ast\left(\sO_{\widetilde{X}^{[N]}_{\widetilde{S}}}(rK_{\widetilde{X}^{[N]}_{\widetilde{S}}/\widetilde{S}} + r\widetilde{\Delta}^{[N]})\right) \otimes \sigma^\ast(I) \subset \sigma^\ast g_{2\ast}\left(\sO_{Y^{(N)}_2}(rK_{Y^{(N)}_2/S} + r\Delta^{(N)}_2 + rg_2^\ast(D))\right).  
		\end{align}
		\emph{Step 3: Final step.} 
		Let
		$$W:=\widetilde{f}_\ast(\sO_{\widetilde{X}}(rK_{\widetilde{X}/\widetilde{S}}+r\widetilde{\Delta}))\quad\textrm{and}\quad l:={\rm rank}(W).$$
		By the construction of $\lambda_{a,r}$, one has
		\begin{align*}
			(\xi\circ\sigma)^\ast\lambda_{a,r}\simeq\det(W).
		\end{align*}
		Let $k\in\bZ^{>0}$ such that $\xi^\ast(\lambda_{a,r})^{\otimes k}$ is a line bundle.
		Now we consider the diagram (\ref{align_fiberprods_complete}) in the case when $N=klr$.
		It follows from Lemma \ref{lem_mild_pushforward} that there is a natural inclusion
		\begin{align*}
			\det(W)^{\otimes kr}\to \left(\widetilde{f}_{\ast}(\sO_{\widetilde{X}}(rK_{\widetilde{X}/\widetilde{S}}+r\widetilde{\Delta}))^{\otimes klr}\right)^{\vee\vee}\simeq\widetilde{f}^{[klr]}_{\ast}(\sO_{\widetilde{X}^{[klr]}}(rK_{\widetilde{X}^{[klr]}/\widetilde{S}}+r\widetilde{\Delta}^{[klr]})).
		\end{align*}
		Combining it with (\ref{align_finally}), we obtain a non-zero map
		$$\sigma^\ast\left(\xi^\ast(\lambda_{a,r})^{\otimes k}\otimes I_Z\right)^{\otimes r}\to\sigma^\ast g_{2\ast}\left(\sO_{Y_2^{(klr)}}(rK_{Y_2^{(klr)}/S}+r\Delta_2^{(klr)}+rg_2^{\ast}(D))\right),$$
		where $I_Z$ is some ideal sheaf on $S$ whose co-support $Z$ lies in the codimension $\geq2$ loci over which $\sigma$ is not flat.
		Taking the adjoint, we have a morphism
		\begin{align*}
			\alpha:\left(\xi^\ast(\lambda_{a,r})^{\otimes k}\otimes I_Z\right)^{\otimes r}\to& \sigma_\ast\sigma^\ast g_{2\ast}\left(\sO_{Y_2^{(klr)}}(rK_{Y_2^{(klr)}/S}+r\Delta_2^{(klr)}+rg_2^{\ast}(D))\right)\\\nonumber
			\stackrel{\rm trace}{\to}&g_{2\ast}\left(\sO_{Y_2^{(klr)}}(rK_{Y_2^{(klr)}/S}+r\Delta_2^{(klr)}+rg_2^{\ast}(D))\right).
		\end{align*}
		By the constructions, $\alpha|_{S^o}$ is the composition of the maps
		\begin{align*}
			\left(\det f^o_\ast(\sO_{X^o}(rK_{X^o/S^o}+r\Delta^o))\right)^{\otimes kr}&\subset f^o_\ast(\sO_{X^o}(rK_{X^o/S^o}+r\Delta^o))^{\otimes klr}\\\nonumber
			&\simeq f^{o[klr]}_\ast\left(\sO_{X^{o[klr]}_{S^o}}(rK_{X^{o[klr]}_{S^o}/S^o}+r\Delta^{o[klr]})\right)\quad \textrm{(Lemma \ref{lem_mild_pushforward}) }\\\nonumber
			&\simeq g_{2\ast}\left(\sO_{Y_2^{(klr)}}(rK_{Y_2^{(klr)}/S}+r\Delta_2^{(klr)}+rg_2^{\ast}(D))\right)|_{S^o}.
		\end{align*}
		Hence $\alpha|_{S^o}$ is injective. By torsion freeness, $\alpha$ is an injective map. Therefore $\alpha$ induces a non-zero morphism
		\begin{align*}
			\left(\xi^\ast(\lambda_{a,r})^{\otimes k}\otimes\sO_S(-D)\otimes I_Z\right)^{\otimes r}\to g_{2\ast}\left(\sO_{Y_2^{(klr)}}(rK_{Y_2^{(klr)}/S}+r\Delta_2^{(klr)})\right).
		\end{align*}
		Applying Theorem \ref{thm_VZ_Arakelov} to the morphism $g_2 \colon (Y_2^{(klr)}, \Delta_2^{(klr)}) \to S$-a simple normal crossing family over $S^\circ$ of relative dimension $klrd$-and to the torsion-free sheaf $\xi^\ast(\lambda_{a,r})^{\otimes k} \otimes \mathcal{O}_S(-D) \otimes I_Z$, we obtain the inequality  
		\begin{align*}  
			c_1\!\left(\xi^\ast(\lambda_{a,r})^{\otimes k} \otimes \mathcal{O}_S(-2D) \otimes I_Z\right) \cdot \alpha \leq \frac{\sum_{p=1}^{klrd} p\,n^{p-1}}{klrd + 1}\, (K_S + D) \cdot \alpha  
		\end{align*}  
		for every movable curve class $\alpha \in N_1(S)$. This yields  
		\begin{align*}  
			c_1\!\left(\xi^\ast(\lambda_{a,r})\right) \cdot \alpha \leq \frac{1}{k(klrd + 1)} \sum_{p=1}^{klrd} p\,n^{p-1}\, (K_S + D) \cdot \alpha + \frac{2}{k}\, D \cdot \alpha,  
		\end{align*}  
		valid for all movable $\alpha \in N_1(S)$ and all positive integers $k$ divisible by $\operatorname{ind}_S(\lambda_{a,r})$. This establishes (\ref{align_Ara_1}).  
		
		If, in addition, $K_S + D$ is ample, then Theorem \ref{thm_VZ_Arakelov} further implies  
		\begin{align*}  
			c_1\!\left(\xi^\ast(\lambda_{a,r})^{\otimes k} \otimes \mathcal{O}_S(-2D) \otimes I_Z\right) \cdot (K_S + D)^{n-1} \leq \frac{klrd}{n}\, (K_S + D)^n  
		\end{align*}  
		for every positive integer $k$ divisible by $\operatorname{ind}_S(\lambda_{a,r})$. Taking $k\to+\infty$ yields (\ref{align_Ara_2}).
	\end{proof}
\begin{cor}\label{cor_Arakelov_dim1}
	Let the notations and assumptions be as in Theorem \ref{thm_Arakelov_family}. Suppose further that $\dim S = 1$. Then  
	\begin{equation*}
		\deg\, \xi^\ast(\lambda_{a,r}) \leq \frac{r n \cdot \operatorname{rank}(\Lambda_{a,r})}{2} \cdot \deg(K_S + D).
	\end{equation*}
\end{cor}
\begin{proof}
	By Theorem \ref{thm_Arakelov_family}, for every positive integer $k$ divisible by $\operatorname{ind}_S(\lambda_{a,r})$, we have  
	\begin{equation*}
		\deg\, c_1\big(\xi^\ast(\lambda_{a,r})\big) 
		\leq \frac{1}{k(k l r n + 1)} \sum_{p=1}^{k l r n} p \cdot \deg(K_S + D) + \frac{2}{k} \deg D
		= \frac{l r n}{2} \deg(K_S + D) + \frac{2}{k} \deg D,
	\end{equation*}
	where $l := \operatorname{rank}(\Lambda_{a,r})$. Letting $k \to +\infty$, the term $\frac{2}{k}\deg D$ vanishes, yielding the claimed inequality.
\end{proof}
    \section{Boundedness of families of stable minimal models}
    \subsection{Birationally admissible condition}
    The admissibility condition plays a central role in the Shafarevich boundedness conjecture for families of varieties. As the following examples demonstrate, boundedness fails in general for arbitrary families of stable minimal models.   
    \begin{exmp}[\emph{Degeneration of the fiber}]\label{exmp_1}
    	Let $f \colon X \to \mathbb{P}^1$ be a Lefschetz pencil whose general fibers are canonically polarized $d$-folds of volume $v$. Denote by $S \subset \mathbb{P}^1$ the set of critical values of $f$, then $\# S \geq 3$ by \cite{VZ2001}. The morphism $f \colon (X,0) \to \mathbb{P}^1$, equipped with the trivial boundary and polarization data, defines a family of $(d,\Phi_0,\{1\},v)$-stable minimal models. Now let $\tau \colon \mathbb{P}^1 \to \mathbb{P}^1$ be any finite morphism, and denote by $f_\tau \colon X \times_{\mathbb{P}^1} \mathbb{P}^1 \to \mathbb{P}^1$ the base change of $f$ along $\tau$. Since $\# \tau^{-1}(S)$ grows without bound as $\deg \tau \to \infty$, the collection $\{f_\tau\}_{\tau}$ cannot be parametrized by any bounded family of $(d,\Phi_0,\{1\},v)$-stable minimal models over $\mathbb{P}^1$.
    \end{exmp}   
    \begin{exmp}[\emph{Degeneration of the polarization}]\label{exmp_2}
    	Let $E$ be an elliptic curve over $\mathbb{C}$, and fix a point $x_0 \in E(\mathbb{C})$. Set $X = E \times E$, and let $f \colon X \to E$ be the projection onto the first factor. Define the $\mathbb{Q}$-divisor $A = \frac{1}{2}\big(E \times \{x_0\} + \Delta_E\big)$ on $X$, where $\Delta_E \subset E \times E$ denotes the diagonal. Then $f \colon (X,A) \to E$ is a family of $(1,\Phi_{1/2},\{1\},t)$-stable minimal models. Although the underlying family of elliptic curves is isotrivial, the polarization divisor $A$ degenerates at $x_0$, rendering the family non-isotrivial as a polarized family. Consider the base changes $f_n \colon X \times_E E \to E$ induced by multiplication-by-$n$ maps $[n] \colon E \to E$. For each $n \geq 1$, the degeneration locus of $f_n$ consists of the $n$ distinct points $\big\{\frac{1}{n}x_0, \dots, \frac{n-1}{n}x_0\big\} \subset E$, and thus becomes arbitrarily large as $n \to \infty$. Consequently, the collection $\{f_n\}_{n \geq 1}$ cannot lie in any bounded family of $(1,\Phi_{1/2},\{1\},t)$-stable minimal models over $E$.
    \end{exmp}
    The notion of birationally admissible family (Definition \ref{defn_admissible}) ensures that both the fibers and the polarizations do not degenerate. This notion naturally generalizes to family over an algebraic stack.
    Let $\sX$ be a reduced, separated Deligne-Mumford stack of finite type over ${\rm Spec}(\bC)$. Let $\Delta$ be an effective $\bQ$-divisor on $\sX$. Let $\pi: X_0 \to \sX$ be an \'etale covering by a scheme $X_0$, and let $\Delta_0 := \pi^\ast(\Delta)$ denote the pullback of $\Delta$ via the \'etale map $\pi$. Such a pair $(X_0, \Delta_0)$ is referred to as a \emph{chart} of $(\sX, \Delta)$.
    
    The pair $(\sX, \Delta)$ is referred to as a \emph{stacky semi-pair} if one of its charts $(X_0, \Delta_0)$ is a semi-pair (Definition \ref{defn_semipair}). In this case, every chart of $(\sX, \Delta)$ is a semi-pair.
    
    Let $\mathcal{S}$ be a smooth Deligne-Mumford stack. A \emph{projective simple normal crossing stacky pair over $\mathcal{S}$} is defined as a representable, projective, and simple normal crossing morphism $\sX \to \mathcal{S}$, together with a simple normal crossing effective $\bQ$-divisor $\Delta$ on $\sX$, such that every stratum of $(\sX,{\rm supp}(\Delta))$ is (representable) smooth over $\mathcal{S}$.
    
    Let $\mathcal{S}$ be a reduced Deligne-Mumford stack. Let $\pi: S_0 \to \mathcal{S}$ be an \'etale cover from a reduced scheme $S_0$. A family of $(d, \Phi_c, \Gamma, \sigma)$-stable minimal models over $\mathcal{S}$ consists of a representable projective morphism $f: \sX \to \mathcal{S}$ and two effective $\bQ$-divisors $\sA$ and $\sB$ on $\sX$, such that the base change family $(X_0, B_0), A_0 \to S_0$, as depicted in the following diagram,  
    $$
    \xymatrix{
    	(X_0, B_0), A_0 \ar[r] \ar[d]^{f_0} & (\sX, \sB), \sA \ar[d]^{f}\\
    	S_0 \ar[r]^{\pi} & \mathcal{S}
    }
    $$  
    is a family of $(d, \Phi_c, \Gamma, \sigma)$-stable minimal models over $S_0$. If $(\sX, \sB), \sA \to \mathcal{S}$ is a family of $(d, \Phi_c, \Gamma, \sigma)$-stable minimal models and $T \to \mathcal{S}$ is any \'etale covering map, then the base change family $(\sX, \sB), \sA \times_{\mathcal{S}} T \to T$ is also a family of $(d, \Phi_c, \Gamma, \sigma)$-stable minimal models.
    \begin{defn}
    	Let $(\sX, \Delta)$ be a stacky semi-pair. A \emph{semi-log resolution} is a representable, projective, and birational morphism $f: (\sX', \Delta') \to (\sX, \Delta)$ such that, for any base change diagram of charts  
    	$$
    	\xymatrix{
    		(X'_0, \Delta'_0) \ar[d]^{\pi'} \ar[r]^g & (X_0, \Delta_0) \ar[d]^{\pi}\\
    		(\sX', \Delta') \ar[r]^f & (\sX, \Delta)
    	},
    	$$  
    	$g$ is a semi-log resolution of semi-pairs (Definition \ref{defn_semiresolution_semipair}).
    \end{defn}
    \begin{defn}
    	Let $(\sX_1, \Delta_1) \to \mathcal{S}$ and $(\sX_2, \Delta_2) \to \mathcal{S}$ be two surjective morphisms from stacky semi-pairs to a smooth Deligne-Mumford stack $\mathcal{S}$ over ${\rm Spec}(\bC)$. We say that $(\sX_1, \Delta_1)$ is strictly log birational to $(\sX_2, \Delta_2)$ over $\mathcal{S}$ if there exists a common semi-log resolution $\pi_i: (\sX', \Delta') \to (\sX_i, \Delta_i)$, $i = 1, 2$, over $\mathcal{S}$, such that the induced morphisms of fibers $\pi_{i,s}: (\sX'_s, \Delta'_s) \to (\sX_{i,s}, \Delta_{i,s})$, $i = 1, 2$, are semi-log resolutions for every geometric point $s \in \mathcal{S}$.  
    	
    	Furthermore, we say that a morphism $f: (\sX, \Delta) \to \mathcal{S}$ admits a \emph{simple normal crossing strictly log birational model} if there exists a simple normal crossing family $f': (\sX', \Delta') \to \mathcal{S}$ that is strictly log birational to $(\sX, \Delta)$ over $\mathcal{S}$.
    \end{defn}
    \begin{defn}\label{defn_st_bir_admissible}
    	Let $f: (\sX, \sB), \sA \to \mathcal{S}$ be a family of $(d, \Phi_c, \Gamma, \sigma)$-stable minimal models over a smooth Deligne-Mumford stack $\mathcal{S}$.  
    	The morphism $f$ is said to be \emph{strictly birationally admissible} if $(\sX, \sB + \sA) \to \mathcal{S}$ admits a simple normal crossing strictly log birational model. A locally closed substack $\mathcal{S}$ of $\mathcal{M}_{\rm slc}(d, \Phi_c, \Gamma, \sigma)$ is called \emph{strictly birationally admissible} if the universal family over $\mathcal{S}$ is strictly birationally admissible.
    \end{defn}
    If $f: (\sX, \sB), \sA \to \mathcal{S}$ is strictly birationally admissible and $S \to \mathcal{S}$ is a morphism from a reduced scheme to $\mathcal{S}$, then the base change of $f$ to $S$ is birationally admissible.
    \subsection{Deformation boundedness}
	Let $\mathcal{S} \subset \mathcal{M}_{\mathrm{slc}}(d, \Phi_c, \Gamma, \sigma)$ be a strictly birationally admissible locally closed substack. Let $S^o$ be an algebraic variety. A family over $S^o$ is said to be $\mathcal{S}$-admissible if all of its fibers belong to $\mathcal{S}$. The set of isomorphism classes of $\mathcal{S}$-admissible families over $S^o$ is in natural bijection with the set $\operatorname{Hom}(S^o, \mathcal{S})$ of morphisms from $S^o$ to $\mathcal{S}$. Following \cite{Kovacs2011}, we define the deformation equivalence relation $\simeq_{\mathcal{S}}$ on $\operatorname{Hom}(S^o, \mathcal{S})$: for $f, g \in \operatorname{Hom}(S^o, \mathcal{S})$, we write $f \simeq_{\mathcal{S}} g$ if there exists a reduced scheme $T$ of finite type over $\mathbb{C}$, a morphism $F \colon S^o \times T \to \mathcal{S}$, and two closed points $x_1, x_2 \in T$ such that $f = F|_{S^o \times \{x_1\}}$ and $g = F|_{S^o \times \{x_2\}}$.
	\begin{thm}\label{thm_main_set_proof}
		Let $S^o$ be an algebraic variety with $S^o_{\rm sing}$ being a compact algebraic subset. Let $\mathcal{S}\subset \sM_{\rm slc}(d, \Phi_c, \Gamma, \sigma)$ be any strictly birationally admissible locally closed substack. Then the set
		$$Hom(S^o,\mathcal{S})/\simeq_{\mathcal{S}}$$
		is a finite set. 
	\end{thm}
    \begin{proof}
    	The proof follows from \cite[Theorem 1.7]{Kovacs2011}: $\mathcal{S}$ is automatically compatible and the weakly boundedness of $\mathcal{S}$ follows from Corollary \ref{cor_Arakelov_dim1}.
    \end{proof}
    \subsection{Boundedness of the Hom stack}
    Let $S$ be a proper algebraic variety and $\mathcal{X}$ a Deligne–Mumford stack of finite type over $\mathbb{C}$. Then the Hom-stack $\operatorname{Hom}(S,\mathcal{X})$ is a Deligne-Mumford stack locally of finite type \cite[Theorem 1.1]{Olsson2006}. Let $Z \subset S$ be a closed algebraic subset and $\mathcal{Z} \subset \mathcal{X}$ a closed algebraic substack. The relative Hom-stack $\operatorname{Hom}((S,Z),(\mathcal{X},\mathcal{Z}))$ is defined as the 2-fiber product  
    $$\xymatrix{
    	Hom((S,Z),(\sX,\sZ)) \ar[r]\ar[d] & Hom(S,\sX) \ar[d]\\
    	Hom(Z,\sZ) \ar[r] & Hom(Z,\sX)
    }.$$ 
    In particular, $\operatorname{Hom}((S,Z),(\mathcal{X},\mathcal{Z}))$ is a closed substack of $\operatorname{Hom}(S,\mathcal{X})$.      
    \begin{thm}\label{thm_main_Hom_proof}  
    	Let $S^o$ be an algebraic variety whose singular locus $S^o_{\mathrm{sing}}$ is a compact algebraic subset. Let $S$ be a compactification of $S^o$, and set $Z := S \setminus S^o$. Let $\mathcal{S} \subset \mathcal{M}_{\mathrm{slc}}(d, \Phi_c, \Gamma, \sigma)$ be a strictly birationally admissible locally closed substack, and denote by $\overline{\mathcal{S}}$ its closure in $\mathcal{M}_{\mathrm{slc}}(d, \Phi_c, \Gamma, \sigma)$. Then the Hom-stack  
    	\[
    	\operatorname{Hom}\big((S,Z),\,(\overline{\mathcal{S}},\,\overline{\mathcal{S}} \setminus \mathcal{S})\big)
    	\]  
    	is of finite type.
    \end{thm}
    \begin{proof}
    	For brevity, denote $\mathcal{M} := \mathcal{M}_{\mathrm{slc}}(d, \Phi_c, \Gamma, \sigma)$ and $M := M_{\mathrm{slc}}(d, \Phi_c, \Gamma, \sigma)$ its coarse moduli space, which is a projective variety. Consider the composition  
    	\[
    	\iota \colon \operatorname{Hom}\big((S,Z),\,(\overline{\mathcal{S}},\,\overline{\mathcal{S}} \setminus \mathcal{S})\big) \to \operatorname{Hom}(S,\mathcal{M}) \to \operatorname{Hom}(S,M).
    	\]  
    	Each $\mathbb{C}$-point $f \in \operatorname{Hom}\big((S,Z),\,(\overline{\mathcal{S}},\,\overline{\mathcal{S}} \setminus \mathcal{S})\big)(\mathbb{C})$ corresponds to a family $f \colon (X,B),A \to S$ that is $\mathcal{S}$-admissible over $S^o$. Hence, by \cite[Proposition 2.6]{Kovacs2011} and Corollary \ref{cor_Arakelov_dim1}, the image of $\iota$ is contained in a finite-type subscheme $\mathbf{H} \subset \operatorname{Hom}(S,M)$. Let $\mathcal{H} \subset \operatorname{Hom}(S,\mathcal{M})$ be the preimage of $\mathbf{H}$, defined via the 2-fiber product  
    	\[
    	\mathcal{H} := \mathbf{H} \times_{\operatorname{Hom}(S,M)} \operatorname{Hom}(S,\mathcal{M}).
    	\]  
    	By \cite[Theorem 1.1]{Olsson2007}, $\mathcal{H}$ is a Deligne-Mumford stack of finite type. Consequently, the inclusion  
    	\[
    	\operatorname{Hom}\big((S,Z),\,(\overline{\mathcal{S}},\,\overline{\mathcal{S}} \setminus \mathcal{S})\big) \hookrightarrow \mathcal{H}
    	\]  
    	implies that the source is itself a stack of finite type.
    \end{proof}
	\bibliographystyle{plain}
	\bibliography{SBC}
\end{document}